\documentclass[11pt,oneside]{article}
\usepackage{amssymb,amsmath,latexsym,amsfonts,amscd,amsthm}
\usepackage[all]{xy}
\usepackage{fancyhdr}
\usepackage[headings]{fullpage}

\newcommand{\gt}[1]{\mathfrak{#1}}
\newcommand{\mc}[1]{\mathcal{#1}}

\newcommand{\RR}{{\mathbb R}}
\newcommand{\RRm}{\RR^{\times}}
\newcommand{\RRp}{\RR^+}
\newcommand{\PP}{{\mathbb P}}
\newcommand{\ZZ}{{\mathbb Z}}
\newcommand{\ZZp}{\ZZ_{>0}}
\newcommand{\QQ}{{\mathbb Q}}
\newcommand{\QQm}{{\mathbb Q}^{\times}}
\newcommand{\QQp}{{\mathbb Q}^+}
\newcommand{\kk}{{\mathbf k}}
\newcommand{\kkm}{\kk^{\times}}
\newcommand{\kkp}{\kk^+}

\newcommand{\lab}{{\langle}}    
\newcommand{\rab}{{\rangle}}    

\newcommand{\End}{\operatorname{End}}

\newcommand{\Fix}{\operatorname{Fix}}

\newcommand{\Span}{\operatorname{Span}}

\newcommand{\Sym}{{\rm Sym}}
\newcommand{\Alt}{{\rm Alt}}
\newcommand{\Pdet}{P\!\det}

\newcommand{\lev}{\operatorname{lev}}
\newcommand{\levo}{\lev_0}
\newcommand{\val}{\operatorname{val}}
\newcommand{\adj}{\operatorname{adj}}

\newcommand{\PSL}{\operatorname{\textsl{PSL}}}    
\newcommand{\SL}{\operatorname{\textsl{SL}}}      
\newcommand{\PGL}{\operatorname{\textsl{PGL}}}    
\newcommand{\GL}{\operatorname{\textsl{GL}}}      
\newcommand{\SU}{\operatorname{\textsl{SU}}}    
\newcommand{\MM}{\mathbb{M}}    
\newcommand{\Co}{\operatorname{\textsl{Co}}}    

\newtheorem{thm}{Theorem}[section]

\newtheorem{lem}[thm]{Lemma}
\newtheorem{prop}[thm]{Proposition}

\theoremstyle{definition}

\theoremstyle{remark}
\newtheorem*{rmk}{Remark}

\numberwithin{equation}{section}

\begin{document}

\title{
    \textsc{Arithmetic groups and the affine $E_8$ Dynkin diagram}
          }

\author{John F. Duncan
          \footnote{
          Harvard University,
          Department of Mathematics,
          One Oxford Street,
          Cambridge, MA 02138,
          U.S.A.}
          {}\footnote{
          Email: {\tt duncan@math.harvard.edu};\;
          homepage: {\tt
          http://math.harvard.edu/\~{}jfd/}
               }
               }

\date{October 14, 2007}

\maketitle

\begin{abstract}
Several decades ago, John McKay suggested a correspondence between
nodes of the affine $E_8$ Dynkin diagram and certain conjugacy
classes in the Monster group. Thanks to Monstrous Moonshine, this
correspondence can be recast as an assignment of discrete subgroups
of $\PSL_2(\RR)$ to nodes of the affine $E_8$ Dynkin diagram. The
purpose of this article is to give an explanation for this latter
correspondence using elementary properties of the group
$\PSL_2(\RR)$. We also obtain a super analogue of McKay's
observation, in which conjugacy classes of the Monster are replaced
by conjugacy classes of Conway's group --- the automorphism group of
the Leech lattice.
\end{abstract}


\section{Introduction}\label{sec:intro} %

Several decades ago, John McKay suggested the following assignment
of conjugacy classes in the Monster group to nodes of the affine
$E_8$ Dynkin diagram (c.f. \cite[\S14]{ConCnstM}).
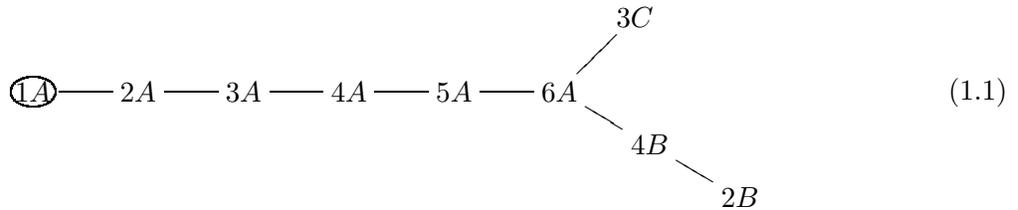
\begin{equation}\label{diag:intro_mnstrcorresp}
     \xy
     (0,0)*+{1A}="1A";
     (0,0)*\xycircle(3,2){-};
     (14,0)*+{2A}="2A";
     (28,0)*+{3A}="3A";
     (42,0)*+{4A}="4A";
     (56,0)*+{5A}="5A";
     (70,0)*+{6A}="6A";
     (80,10)*+{3C}="3C";
     (82,-7)*+{4B}="4B";
     (94,-14)*+{2B}="2B";
     {\ar@{-} "1A";"2A"};
     {\ar@{-} "2A";"3A"};
     {\ar@{-} "3A";"4A"};
     {\ar@{-} "4A";"5A"};
     {\ar@{-} "5A";"6A"};
     {\ar@{-} "6A";"3C"};
     {\ar@{-} "6A";"4B"};
     {\ar@{-} "4B";"2B"};
     \endxy
\end{equation}
This has become known as {\em McKay's Monstrous $E_8$ observation}.
Since elements of the Monster group determine principal moduli for
genus zero subgroups of $\PSL_2(\RR)$ via Monstrous Moonshine (c.f.
\cite{ConNorMM},\cite{BorMM}), the assignment
(\ref{diag:intro_mnstrcorresp}) entails a correspondence between
nodes of the affine $E_8$ Dynkin diagram and certain discrete
subgroups\footnote{The notation in (\ref{diag:intro_mdlrcorresp})
for subgroups of $\PSL_2(\RR)$ follows \cite{ConNorMM},
\cite{ConMcKSebDiscGpsM}; we write $n+$ for $\Gamma_0(n)+$, and
$n\|h+$ for $\Gamma_0(n\|h)+$, for example.} of $\PSL_2(\RR)$ that
are commensurable with $\PSL_2(\ZZ)$.
\begin{equation}\label{diag:intro_mdlrcorresp}
     \xy
     (0,0)*+{1}="1A";
     (0,0)*\xycircle(2,2){-};
     (14,0)*+{2+}="2A";
     (28,0)*+{3+}="3A";
     (42,0)*+{4+}="4A";
     (56,0)*+{5+}="5A";
     (70,0)*+{6+}="6A";
     (80,10)*+{3\|3}="3C";
     (82,-7)*+{4\|2+}="4B";
     (94,-14)*+{2}="2B";
     {\ar@{-} "1A";"2A"};
     {\ar@{-} "2A";"3A"};
     {\ar@{-} "3A";"4A"};
     {\ar@{-} "4A";"5A"};
     {\ar@{-} "5A";"6A"};
     {\ar@{-} "6A";"3C"};
     {\ar@{-} "6A";"4B"};
     {\ar@{-} "4B";"2B"};
     \endxy
\end{equation}
In this article we furnish a prescription for recovering this latter
correspondence (\ref{diag:intro_mdlrcorresp}), that is given solely
in terms of elementary properties of the group $\PSL_2(\RR)$.

For the purposes of this article, we will say that a subgroup
$\Gamma<\PSL_2(\RR)$ is {\em arithmetic} if it is commensurable with
$\PSL_2(\ZZ)$ in the strong sense; viz. if the intersection of
$\Gamma$ with $\PSL_2(\ZZ)$ has finite index in each of $\Gamma$ and
$\PSL_2(\ZZ)$.

In \S\ref{sec:arith} we set up a general context for studying
arithmetic subgroups of $\PSL_2(\RR)$. The approach here is closely
modeled on that described in \cite{Con_UndrstndgGamma0N}. In
\S\ref{sec:dyn:nodes} we furnish arithmetic conditions that are
satisfied only by the arithmetic groups appearing in
(\ref{diag:intro_mdlrcorresp}). Then in \S\ref{sec:dyn:edges} we
show how the arithmetic properties of these groups can be used to
recover also the edges of the affine $E_8$ Dynkin diagram
(\ref{diag:intro_mdlrcorresp}).

When considering the correspondence (\ref{diag:intro_mdlrcorresp}),
it is hard not to be reminded of {\em the McKay correspondence},
which may be regarded as giving a prescription for recovering the
extended Dynkin diagrams from representations of finite subgroups of
$\SU(2)$ (c.f. \cite{McK_GrphsSingsFGps}). In particular, the McKay
correspondence indicates how to recover the affine $E_8$ Dynkin
diagram from the category of representations of the binary
icosahedral group, $2.\Alt_5\subset\SU(2)$. Our prescription for
recovering the edges of (\ref{diag:intro_mdlrcorresp}) is
reminiscent of this procedure.

In \S\ref{sec:super} we demonstrate how to obtain a kind of ``super
analogue'' of McKay's Monstrous $E_8$ observation, in which
conjugacy classes of Conway's group $\Co_0$ take on the r\^ole
played by conjugacy classes of the Monster in
(\ref{diag:intro_mnstrcorresp}).

\section{Arithmetic groups}\label{sec:arith} %

Let us agree to say that $\Gamma<\PSL_2(\RR)$ is an {\em arithmetic}
subgroup of $\PSL_2(\RR)$ if the intersection
$\PSL_2(\ZZ)\cap\Gamma$ has finite index both in $\PSL_2(\ZZ)$ and
in $\Gamma$. In order to study the arithmetic subgroups of
$\PSL_2(\RR)$ we will adopt the approach of Conway
\cite{Con_UndrstndgGamma0N}, whereby one analyzes these groups in
terms of their actions on lattices. We review this approach in the
present section. Actually, our exposition will employ a slightly
different language to that of \cite{Con_UndrstndgGamma0N}; the story
is nonetheless the same. As a small supplement to the ideas of
\cite{Con_UndrstndgGamma0N}, we include, in \S\ref{sec:arith:cusps},
a description of how to identify the cusps of certain arithmetic
subgroups of $\PSL_2(\RR)$, in terms of such a group's action on
lattices.

\subsection{Orientation}\label{sec:arith:orient}

Let $\kk$ be an ordered field. Write $\kkm$ for the multiplicative
group of non-zero elements of $\kk$, and write $\kkp$ for its
subgroup of {\em positive elements}. Note that an ordered field is
necessarily of characteristic $0$, so there is a natural embedding
$\QQ\hookrightarrow \kk$.

Let $V$ be a vector space of dimension $n$ over $\kk$. We may
consider the {\em $n$-th exterior power} $\wedge^nV$. This vector
space has a distinguished point; viz. the origin, and we may
consider the complement
$(\wedge^nV)^{\times}=\wedge^nV\setminus\{0\}$. The set
$(\wedge^nV)^{\times}$ is naturally a $\kkm$-torsor. Upon picking
any non-zero element, $w$ say, of $\wedge^nV$, we arrive at a
decomposition of $(\wedge^nV)^{\times}$ into two disjoint
$\kkp$-torsors: $(\wedge^nV)^{\times}=\kkp\cdot w\cup\kkp\cdot(-w)$.
We define an {\em orientation} for $V$ to be a choice of
$\kkp$-torsor in $\wedge^nV$. An {\em oriented vector space over
$\kk$} is a pair $(V,W)$ where $V$ is a vector space over $\kk$, and
$W$ is a $\kkp$-torsor in $\wedge^nV$.

Consider the canonical linear map $\triangle:V^n\to\wedge^nV$ given
by
\begin{gather}\label{eqn:wedgemap}
     (v_1,\ldots, v_n)
     \mapsto
     \triangle (v_1,\ldots, v_n):=v_1\wedge\cdots\wedge v_n.
\end{gather}
The set $\{v_1,\ldots,v_n\}$ underlying the $n$-tuple
$(v_1,\ldots,v_n)$ is a linearly independent subset of $V$ if and
only if $\triangle (v_1,\ldots,v_n)\neq 0$. We say that an $n$-tuple
$(v_1,\ldots,v_n)$ is an {\em ordered basis} for $V$ if
$\triangle(v_1,\ldots,v_n)\neq 0$. We write $\mc{B}$ for the set of
ordered bases for $V$.
\begin{gather}
     \mc{B}:=\triangle^{-1}((\wedge^nV)^{\times})\subset V^n
\end{gather}
For $(V,W)$ an oriented vector space,  we say that an ordered set
$(v_1,\ldots,v_n)\in V^n$ is an {\em oriented basis} for $V$ if
$\triangle (v_1,\ldots,v_n)\in W$. Let us write $\mc{B}^+$ for the
set of oriented bases in $V$.
\begin{gather}
     \mc{B}^+:=\triangle^{-1}(W)\subset V^n
\end{gather}

Let $\End(V)$ denote the $\kk$-algebra of $\kk$-linear
transformations of $V$. We may regard $\End(V)$ as acting on $V$
from the right.
\begin{gather}
     \begin{split}
          V\times \End(V)&\to V\\
          (v,A)&\mapsto v\cdot A
     \end{split}
\end{gather}
This action extends naturally to an (right) action of $\End(V)$ on
$V^n$.
\begin{gather}
     (v_1,\ldots,v_n)\cdot A:=(v_1\cdot A,\ldots,v_n\cdot A)
\end{gather}

As is usual, we take $\GL(V)$ to be the complement of $\det^{-1}(0)$
in $\End(V)$, where $\det:\End(V)\to\kk$ is the unique map such that
\begin{gather}
     \triangle({\bf v}\cdot A)=\det(A)\triangle({\bf v})
\end{gather}
for all ${\bf v}\in V^n$ and $A\in\End(V)$. We take $\SL(V)$ to be
the kernel of the restriction $\det:\GL(V)\to\kkm$, and we set
$\GL^+(V)$ to be the preimage of $\kkp$ under this map.
\begin{align}
     \GL(V)&={\det}^{-1}(\kkm)\\
     \GL^+(V)&={\det}^{-1}(\kkp)\\
     \SL(V)&={\det}^{-1}(1)
\end{align}
Observe that the sets $\mc{B}$ and $\mc{B}^+$ are naturally torsors
for $\GL(V)$ and $\GL^+(V)$, respectively.

For any choice of ordered basis ${\bf v}=(v_1,\ldots, v_n)\in\mc{B}$
there is a corresponding isomorphism $\phi_{{\bf
v}}:\GL(V)\to\GL_n(\kk)$ given by setting $\phi_{{\bf v}}(A)$ to be
the matrix $(a^i_j)\in\GL_n(\kk)$ such that $v_j\cdot A=\sum_i
v_ia_j^i$ for $1\leq j\leq n$. Thus we obtain a right action of
$\GL_n(\kk)$ on $V^n$ subject to a choice of ordered basis ${\bf
v}\in \mc{B}$. Observe that there is also a canonical (left) action
of $\GL_n(\kk)$ on $V^n$ given by setting
\begin{gather}
     (a^i_j)\cdot (v_1,\ldots, v_n):=(v_1',\ldots,v_n')
\end{gather}
where $v_j'=\sum_ia^i_jv_i$. Furthermore, this action preserves the
subset $\mc{B}\subset V^n$; indeed, $\GL_n(\kk)$ acts simply
transitively on $\mc{B}$, so that $\mc{B}$ is naturally a (left)
$\GL_n(\kk)$-torsor. The set $\mc{B}^+$, consisting of oriented
bases for $V$, is naturally a (left) torsor for $\GL_n^+(\kk)$.

\subsection{Lattices}\label{sec:arith:latts}

Let $V=(V,W)$ be an oriented vector space of dimension $n$ over
$\kk$, as in \S\ref{sec:arith:orient}. Write $\mc{B}^+$ for the
subset of $V^n$ consisting of oriented bases for $V$.

A {\em lattice} in $V$ is an additive subgroup, $L$ say, of $V$,
such that $L$ is equivalent to $\ZZ^n$ as a $\ZZ$-module, and such
that $\Span_{\kk}L$ (the $\kk$-linear span of $L$) is $V$. Let us
write $\mc{L}$ for the set of all lattices in $V$. Suppose that
$L\in\mc{L}$, and let $\{v_1,\ldots,v_n\}\subset V$ be a set of
generators for $L$ as a $\ZZ$-module, so that
\begin{gather}
     L=\ZZ v_1+\cdots+\ZZ v_n.
\end{gather}
Then $\{v_1,\ldots,v_n\}$ is a linearly independent subset of $V$,
for if $\sum a_iv_i=0$ for some $a_i\in\kk$, then the $\kk$-linear
span of $L$ is vector space of dimension less than $n$,
contradicting the property that $\Span_{\kk}L=V$. Conversely, if
$\{v_1,\ldots,v_n\}$ is a linearly independent subset of $V$ then
the additive subgroup of $V$ generated by the $v_i$ is a lattice in
$V$. We conclude that there is a natural surjective map from the set
of bases for $V$ to the set of lattices in $V$.
\begin{gather}
     \{v_1,\ldots,v_n\}\mapsto \ZZ v_1+\cdots +\ZZ v_n
\end{gather}
Consider the (pre)composition of this map with the natural map
$\mc{B}^+\to\mc{P}(V)$, sending an $n$-tuple in $\mc{B}^+$ to it's
underlying set. Since any basis for $V$ can be made an oriented
basis by equipping it with a suitable ordering, we conclude that
this composition furnishes a (natural) surjective map from
$\mc{B}^+$ to the set of lattices in $V$.
\begin{gather}
     \begin{split}
          \mc{B}^+&\to\mc{L}\\
          (v_1,\ldots,v_n)&\mapsto \ZZ v_1+\cdots +\ZZ v_n
     \end{split}
\end{gather}
We would like to know when two oriented bases for $V$ determine the
same lattice. Recall that $\mc{B}^+$ is naturally a $\GL_n^+(\kk)$
torsor (c.f.\S\ref{sec:arith:orient}). If
\begin{gather}\label{eqn:arith:latts|idlatts}
     \ZZ v_1+\cdots +\ZZ v_n=\ZZ v_1'+\cdots +\ZZ v_n'
\end{gather}
then there are some integers $m^i_{j}$ and $n^i_j$ such that
$v_j'=\sum_i m^i_jv_i$ and $v_j=\sum_i n^i_jv'_i$ for $1\leq j\leq
n$, and hence the matrices $(m^i_j)$ and $(n^i_j)$ are mutually
inverse. In other words, if two oriented bases determine the same
lattice then they are related, via left-multiplication, by an
element of $\SL_n(\ZZ)$. Conversely, if $A\in\SL_n(\ZZ)$ and
$(v_i)\in \mc{B}^+$, then we have the equality
(\ref{eqn:arith:latts|idlatts}) once we set $(v_i')=(A\cdot v_i)$.

We have shown that the set $\mc{L}$ is naturally isomorphic to the
orbit space $\SL_n(\ZZ)\backslash\mc{B}^+$, where $\mc{B}^+$ is
naturally a torsor for $\GL_n^+(\kk)$. Given a choice of oriented
basis ${\bf v}\in \mc{B}^+$, we obtain an identification of
$\mc{B}^+$ with $\GL_n^+(\kk)$
\begin{gather}
     \begin{split}
          \GL_n^+(\kk)&\leftrightarrow \mc{B}^+\\
          A&\leftrightarrow A\cdot {\bf v}
     \end{split}
\end{gather}
and hence, an identification of $\mc{L}$ with
$\SL_n(\ZZ)\backslash\GL_n^+(\kk)$.

From now on we will identify $\mc{L}$ with the orbit space
$\SL_n(\ZZ)\backslash\mc{B}^+$

\subsection{Projective lattices}\label{sec:arith:proj}

Let $V$, $\mc{B}^+$, and $\mc{L}$ be as is \S\ref{sec:arith:latts}.
For simplicity of exposition, let us assume that $n=\dim V$ is even
--- the case that $n$ is odd requires one to occasionally replace
$\kkm$ with $\kkp$.

There is a natural embedding (of groups) $\kkm\hookrightarrow
\GL_n^+(\kk)$. The image of $\kkm$ under this map is central
(indeed, this $\kkm$ is the center of $\GL_n^+(\kk)$), so there is a
natural action of $\kkm$ on the space of lattices,
$\mc{L}=\SL_n(\ZZ)\backslash\mc{B}^+$.
\begin{gather}
     \alpha\cdot\SL_n(\ZZ)(v_1,\ldots,v_n):=\SL_n(\ZZ)(\alpha
     v_1,\ldots,\alpha v_n)
\end{gather}
We wish to consider the quotient space
$P\mc{L}=\kkm\backslash\mc{L}$. That is, we would like to regard two
lattices as equivalent if one is a non-zero scalar multiple of the
other; we call the elements of $P\mc{L}$ the {\em projective
lattices} in $V$.

If $X$ is a set equipped with an action of $\kkm$, we write $PX$ for
the orbit space $\kkm\backslash X$. We write $x\mapsto [x]$ for the
natural map $X\to PX$.

Just as $\mc{L}$ is naturally identified with the orbit space
$\SL_n(\ZZ)\backslash\mc{B}^+$. The set $P\mc{L}$ is naturally
identified with the orbit space $\PSL_n(\ZZ)\backslash P\mc{B}^+$,
where
\begin{gather}
     \PSL_n(\ZZ)=\SL_n(\ZZ)/(\kkm\cap\ZZ),\\
     P\mc{B}^+=\kkm\backslash\mc{B}^+.
\end{gather}
We call $P\mc{B}^+$ the set of {\em projective oriented bases} for
$V$. It is naturally a torsor for $\PGL_n^+(\kk)=\GL_n^+(\kk)/\kkm$,
so that after choosing an element $[{\bf v}]\in P\mc{B}^+$ we obtain
an identification of $\PGL_n^+(\kk)$ with $P\mc{B}^+$ by setting
\begin{gather}
     \begin{split}
          \PGL_n^+(\kk)&\leftrightarrow P\mc{B}^+\\
          [A]&\leftrightarrow [A]\cdot [{\bf v}]:=[A{\bf v}],
     \end{split}
\end{gather}
and hence, also, an identification of
$\PSL_n(\ZZ)\backslash\PGL_n^+(\kk)$ with $P\mc{L}$.

\subsection{Commensurable lattices}\label{sec:arith:commens}

Let $V$, $\mc{B}^+$, and $\mc{L}$ be as in \S\ref{sec:arith:latts}.
Let us chose a lattice $L_1\in \mc{L}$, and let ${\bf
v}_1\in\mc{B}^+$ satisfy $L_1=\SL_n(\ZZ){\bf v}_1$. We say that a
lattice $L$ in $V$ is {\em commensurable} with $L_1$ if the
intersection $L\cap L_1$ has finite index both in $L$ and in $L_1$.
Let $V_{1}$ be the $\QQ$-linear span of $L_1$ in $V$, so that $V_1$
is a vector space of dimension $n$ over $\QQ$ that is contained in
$V$, and the $\kk$-linear span of $V_1$ is $V$.

The following result is straightforward.
\begin{prop}\label{prop:arith:commens|commenslatts}
The lattices in $V$ that are commensurable with $L_1$ are exactly
the additive subgroups of $V_1$ that are equivalent to $\ZZ^n$ as
$\ZZ$-modules.
\end{prop}
The rational vector space $V_1$, determined by $L_1$, inherits an
orientation from $V$ (in the sense of \S\ref{sec:arith:orient})
which we denote by $W_1$. We write $\mc{B}^+_1$ for the set of
oriented bases for $V_1$ (c.f. \S\ref{sec:arith:latts}), and we
write $\mc{L}_1$ for the orbit space $\SL_n(\ZZ)\backslash
\mc{B}^+_1$, which is naturally identified with the set of lattices
in $V_1$ --- in other words, by Proposition
\ref{prop:arith:commens|commenslatts}, $\mc{L}_1$ is the subset of
$\mc{L}$ consisting of the lattices in $V$ that are commensurable
with $L_1$. Our choice ${\bf v}_1\in\mc{B}^+_1$ allows us to
construct an identification of $\mc{L}_1$ with
$\SL_n(\ZZ)\backslash\GL_n^+(\QQ)$ (c.f. \S\ref{sec:arith:latts}).

We will require to distinguish the lattices in $V_1$ only up to
their orbits under $\QQm$. Similar to \S\ref{sec:arith:proj}, we
write $P\mc{L}_1$ for the orbit space $\QQm\backslash \mc{L}_1$.
Then $P\mc{L}_1$ is naturally identified with $\PSL_n(\ZZ)\backslash
P\mc{B}^+_1$ for $P\mc{B}^+_1=\QQm\backslash \mc{B}^+_1$. The choice
${\bf v}_1\in\mc{B}^+_1$ allows us to identify $P\mc{B}^+_1$ with
$\PSL_n(\ZZ)\backslash \PGL_n^+(\QQ)$ (c.f. \S\ref{sec:arith:proj}).

We will continue to write $x\mapsto [x]$ for the projection maps
$X\to PX$. It should be clear from the context whether we are taking
orbits with respect to actions by $\QQm$ or $\kkm$.

\subsection{Hyperdistance}\label{sec:arith:hdist}

We continue in the setting of \S\ref{sec:arith:commens}. In
particular, we continue to assume that $n=\dim V$ is even.

There is a canonically defined integer valued function on
$M_n(\QQ)$, which we call the {\em rational projective determinant},
and denote $\Pdet$, which is defined as follows. If $A=(a^i_j)\in
M_n(\QQ)$ is non-zero, then there is a smallest positive rational
number $\alpha_A\in\QQp$ such that $\alpha_A a^i_j\in\ZZ$ for all
$1\leq i,j\leq n$; explicitly, if we write each $a^i_j$ as a
quotient $a^i_j=b^i_j/c^i_j$ for some $b^i_j,c^i_j\in\ZZ$, then
$\alpha_A= \,{\rm gcd}\{c^i_j\}/{\rm lcm}\{b^i_j\}$. We set
\begin{gather}
     \Pdet(A):=
          \begin{cases}
               \det(\alpha_AA)=\alpha_A^n\det(A),\;&\text{if $A\neq
               0$;}\\
               0,\;&\text{if $A=0$.}
          \end{cases}
\end{gather}
Evidently, $\Pdet(\alpha A)=\Pdet (A)$ for any $\alpha\in \QQm$ (so
long as $n$ is even), so that the projective determinant induces a
well defined function,
\begin{gather}
     \Pdet:PM_n(\QQ)\to \ZZ,
\end{gather}
which we also call the {\em rational projective determinant}, where
$PM_n(\QQ)=\QQm\backslash M_n(\QQ)$ is the monoid of {\em rational
projective matrices}.
\begin{lem}\label{lem:arith:hdist|slninv}
If $A\in\SL_n(\ZZ)$ then $\Pdet(AX)=\Pdet(X)=\Pdet(XA)$ for any
$X\in M_n(\QQ)$.
\end{lem}
\begin{proof}
We have $\Pdet(AX)=\alpha^n_{AX}\det(AX)$ by definition, and
$\det(AX)=\det(X)$ for $A\in \SL_n(\ZZ)$, so it suffices to show
that $\alpha_{AX}=\alpha_X$ for $A\in \SL_n(\ZZ)$. Since both $A$
and its inverse have integer entries, $\alpha AX$ belongs to
$M_n(\ZZ)$ if and only if $\alpha X$ does, for any $\alpha\in \QQp$.
Thus the sets $\{\alpha\in \QQp\mid \alpha AX\in M_n(\ZZ)\}$ and
$\{\alpha\in\QQp\mid \alpha X\in M_n(\ZZ)\}$ coincide, and thus
their minimal elements coincide.
\end{proof}
We will be most interested in the restriction of $\Pdet$ to the
group $\PGL_n^+(\QQ)$, where it takes only positive integer values.
The following result is an immediate consequence of Lemma
\ref{lem:arith:hdist|slninv}.
\begin{prop}\label{prop:arith:hdist|pdetinv}
The projective determinant induces a well defined positive integer
valued function on the orbit space
$\PSL_n(\ZZ)\backslash\PGL_n^+(\QQ)$.
\begin{gather}
     \Pdet:\PSL_n(\ZZ)\backslash\PGL_n^+(\QQ)\to \ZZp
\end{gather}
This function is invariant under the natural right action of
$\PSL_n(\ZZ)$.
\end{prop}
In particular then, the projective determinant furnishes us with a
natural method for comparing projective lattices in $V_1$. For
suppose given a pair of projective lattices $L,L'\in P\mc{L}_1$. We
regard $P\mc{L}_1$ as identified with $\PSL_n(\ZZ)\backslash
P\mc{B}^+_1$, so that we have
\begin{gather}\label{eqn:arith:hdist|Lsetup}
     L=\PSL_n(\ZZ)\cdot [{\bf v}],\quad
     L'=\PSL_n(\ZZ)\cdot [{\bf v}'],
\end{gather}
for some $[{\bf v}],[{\bf v}']\in P\mc{B}^+_1$. The set
$P\mc{B}^+_1$ is a $\PGL_n^+(\QQ)$-torsor, so there is a unique
$g\in\PGL_n^+(\QQ)$ such that $g\cdot [{\bf v}]=[{\bf v}']$. We
define a positive integer $\delta(L,L')$ by setting
\begin{gather}\label{eqn:arith:hdist|distdef}
     \delta(L,L')=\Pdet(g).
\end{gather}
\begin{prop}
The function $\delta: P\mc{L}\times P\mc{L}\to\ZZp$ given by
(\ref{eqn:arith:hdist|distdef}) is well-defined.
\end{prop}
\begin{proof}
We should check that the value of $\delta(L,L')$ does not depend
upon the choice of representatives (viz. $[{\bf v}]$ and $[{\bf
v}']$), for the orbits $L$ and $L'$. So let $L,L'\in P\mc{L}_1$ be
as in (\ref{eqn:arith:hdist|Lsetup}), with $g\in \PGL_n^+(\QQ)$
satisfying $g\cdot [{\bf v}]=[{\bf v}']$, and suppose that
$L=\PSL_n(\ZZ)\cdot[{\bf w}]$ and $L'=\PSL_n(\ZZ)\cdot [{\bf w}']$.
Then there are some $h,h'\in\PSL_n(\ZZ)$ such that $[{\bf w}]=h\cdot
[{\bf v}]$ and $[{\bf w}']=h'\cdot [{\bf v}']$, and $h'gh^{-1}$ is
the unique element of $\PGL_n^+(\QQ)$ satisfying $(h'gh^{-1})\cdot
[{\bf w}]=[{\bf w}']$. We require to show that
$\Pdet(g)=\Pdet(h'gh^{-1})$, but this is guaranteed by Proposition
\ref{prop:arith:hdist|pdetinv}.
\end{proof}
In the case that $n=2$ the function $\delta$ acquires a special
property.
\begin{prop}\label{prop:arith:hdist|symm}
For $n=2$, the function $\delta: P\mc{L}\times P\mc{L}\to \ZZp$ is
symmetric.
\end{prop}
\begin{proof}
If $X\in M_n(\ZZ)$ and $X$ is invertible, then $\det(X)X^{-1}\in
M_n(\ZZ)$. Applying this same rule with $\det(X)X^{-1}$ in place of
$X$, we see that
\begin{gather}
          \frac{\det(X)^{n}}{\det(X)}\frac{X}{\det(X)}\in
          M_n(\ZZ)
\end{gather}
Thus, in the special case that $n=2$, we have that an invertible
matrix $X$ lies in $M_n(\ZZ)$ if and only if $\det(X)X^{-1}$ does.
Consequently, for any $A\in\GL_2^+(\QQ)$ and any $\alpha\in\QQp$ we
have that $\alpha A$ belongs to $M_2(\ZZ)$ if and only if
$\alpha\det(A)A^{-1}$ does, so that if $\alpha_A$ is the minimal
positive rational for which $\alpha_AA$ belongs to $M_2(\ZZ)$, then
$\alpha_{A^{-1}}=\det(A)\alpha_A$ is the minimal positive rational
that has this property for $A^{-1}$. It is easy to compute now that
$\Pdet(A)=\alpha_A^2\det(A)$ and
$\Pdet(A^{-1})=\alpha_{A^{-1}}^2\det(A^{-1})$ coincide. This
verifies the claim, since if $\delta(L,L')=\Pdet(g)$ then
$\delta(L',L)=\Pdet(g^{-1})$.
\end{proof}
\begin{rmk}
It is easy to find $g\in\PGL_n^+(\QQ)$ such that $\Pdet(g)\neq
\Pdet(g^{-1})$ when $n>2$.
\end{rmk}
Following Conway \cite{Con_UndrstndgGamma0N}, we call the function
$\delta$, when defined on pairs of projective lattices in a rational
vector space of dimension $2$, {\em hyperdistance}. The logarithm of
$\delta$ is in fact a metric on $P\mc{L}$ for $n=2$.

From now on we will restrict attention to the case that $n=\dim
V=2$.

\subsection{Names}\label{sec:arith:names}

We continue in the setting of \S\ref{sec:arith:hdist}. In
particular, we assume that $n=\dim V=2$, and we retain our choice of
lattice $L_1\in \mc{L}$, and generating set ${\bf
v_1}\in\mc{B}^+_1$. We regard the set $P\mc{L}_1$, of projective
lattices in $V_1$, as identified with the coset space
$\PSL_2(\ZZ)\backslash\PGL_2^+(\QQ)$ via this choice.
\begin{gather}\label{eqn:arith:names|ident}
     \begin{split}
     \PSL_2(\ZZ)\backslash\PGL_2^+(\QQ)&\leftrightarrow
          P\mc{L}_1\\
     \PSL_2(\ZZ)[A]&\leftrightarrow \PSL_2(\ZZ)\cdot [A{\bf v}_1]
     \end{split}
\end{gather}
Of course, $\PGL_2^+(\QQ)$ acts naturally, from the right, on
$\PSL_2(\ZZ)\backslash\PGL_2^+(\QQ)$, so the identification
(\ref{eqn:arith:names|ident}) also allows us to define a right
action of $\PGL_2^+(\QQ)$ on $P\mc{L}_1$. Evidently, hyperdistance
(c.f. \S\ref{sec:arith:hdist}) is invariant with respect to this
$\PGL_2^+(\QQ)$-action.

We will now introduce {\em names} for the elements of the coset
space
\begin{gather}\label{eqn:arith:names|PZbkslshPQ}
     \PSL_2(\ZZ)\backslash\PGL_2^+(\QQ),
\end{gather}
and hence also for the projective lattices in $V_1$. Consider the
natural map from $\GL_2^+(\QQ)$ to
$\PSL_2(\ZZ)\backslash\PGL_2^+(\QQ)$. Our strategy (following
\cite{Con_UndrstndgGamma0N}) is to try and define a canonical
preimage in $\GL_2^+(\QQ)$ for each coset $\PSL_2(\ZZ)g$ in the
image of this map.

Let $g\in\PGL_2^+(\QQ)$. We will write
\begin{gather}
     \left(
       \begin{array}{cc}
         a & b \\
         c & d \\
       \end{array}
     \right)\mapsto
     \left[
       \begin{array}{cc}
         a & b \\
         c & d \\
       \end{array}
     \right]
\end{gather}
for the canonical map $M_2(\QQ)\to PM_2(\QQ)$. Then
\begin{gather}
     g=
     \left[
       \begin{array}{cc}
         a & b \\
         c & d \\
       \end{array}
     \right]
\end{gather}
for some $a,b,c,d\in\QQ$ such that $ad-bc>0$. Choose integers
$s,t\in\ZZ$ such that $sa+tc=0$, and observe that then $sb+td\neq
0$, by linear independence of the columns of an invertible matrix.
We may assume that $s,t$ have no common factors, and thus there are
integers $m,n\in\ZZ$ such that $mt-sn=1$. In other words, there is
an $h\in\PSL_2(\ZZ)$ such that
\begin{gather}
     h\cdot g=g'=
     \left[
       \begin{array}{cc}
         a' & b' \\
         0 & d' \\
       \end{array}
     \right]
\end{gather}
for some $a',b',d'\in\QQ$, with $d'\neq 0$. We have
\begin{gather}
     \left[
       \begin{array}{cc}
         a' & b' \\
         0 & d' \\
       \end{array}
     \right]
     =
     \left[
       \begin{array}{cc}
         a'' & b'' \\
         0 & 1 \\
       \end{array}
     \right]
\end{gather}
for $a''=a'/d'>0$ and $b''=b'/d'$. There is a unique integer, $N$
say, such that $0\leq b''+N<1$. Left-multiplying by
\begin{gather}
     \left[
       \begin{array}{cc}
         1 & N \\
         0 & 1 \\
       \end{array}
     \right],
\end{gather}
we arrive at an element $g'''$ of $\PSL_2(\ZZ)g$ of the form
\begin{gather}
     g'''=
     \left[
       \begin{array}{cc}
         a''' & b''' \\
         0 & 1 \\
       \end{array}
     \right]
\end{gather}
where $a'''>0$ and $0\leq b'''<1$. Let us write $\mc{M}$ for the set
of matrices in $M_2(\QQ)$ of the form
\begin{gather}
     \left(
       \begin{array}{cc}
         M & b \\
         0 & 1 \\
       \end{array}
     \right)
\end{gather}
where $M>0$ and $0\leq b<1$. We have established the following
result.
\begin{prop}
The assignment
\begin{gather}
     \left(
       \begin{array}{cc}
         M & b \\
         0 & 1 \\
       \end{array}
     \right)
     \mapsto
     \PSL_2(\ZZ)
     \left[
       \begin{array}{cc}
         M & b \\
         0 & 1 \\
       \end{array}
     \right]
\end{gather}
defines a bijective correspondence between the elements of $\mc{M}$
and the elements of the coset space
$\PSL_2(\ZZ)\backslash\PGL_2^+(\QQ)$.
\end{prop}
Given $M\in \QQp$ and $b\in \QQ\cap[0,1)$, let us write $g_{M,b}$
for the projective matrix
\begin{gather}\label{eqn:arith:names|gs}
     g_{M,b}:=
     \left[
       \begin{array}{cc}
         M & b \\
         0 & 1 \\
       \end{array}
     \right].
\end{gather}
Let us write $L_{M,b}$ for the corresponding projective lattice with
respect to the identification (\ref{eqn:arith:names|ident}); i.e. we
set
\begin{gather}
     L_{M,b}:=\PSL_2(\ZZ)g_{M,b}.
\end{gather}
We will typically abbreviate $g_{M,0}$ to $g_{M}$, and $L_{M,0}$ to
$L_{M}$. Observe that, under this notational convention, we have
\begin{gather}
     L_{1}=L_{1,0}=\PSL_2(\ZZ)g_{1}
          =\PSL_2(\ZZ)\leftrightarrow \PSL_2(\ZZ)[{\bf v}_1]
\end{gather}
in agreement with our notation $L_1$ for the lattice we picked in
\S\ref{sec:arith:commens}.

Observe that to compute the hyperdistance $\delta(L_{M,b},L_1)$,
from an arbitrary lattice $L_{M,b}$ to the distinguished lattice
$L_1$, is the same as computing the projective determinant of the
matrix $g_{M,b}$. More generally, computing the hyperdistance
between $L_{M',b'}$ and $L_{M,b}$ is the same as computing the
projective determinant of $g_{M',b'}g_{M,b}^{-1}$.
\begin{gather}
     (g_{M',b'}g_{M,b}^{-1})\cdot(g_{M,b}\cdot[{\bf v}_0])
          =g_{M',b'}\cdot[{\bf v}_0]\\
     \delta(L_{M',b'},L_{M,b})=\Pdet(g_{M',b'}g_{M,b}^{-1})
\end{gather}

Our names $L_{M,b}$ for the projective lattices in $V_1$ are not
particularly canonical. We could, for example, have chosen to seek
coset representatives for $\PSL_2(\ZZ)\backslash \PSL_2^+(\QQ)$ with
vanishing top-right entry, unital top-left entry, positive
bottom-right entry, and bottom-left entry in $\QQ\cap[0,1)$. That
is, we could have sought representatives in the form $\bar{g}_{b,M}$
where
\begin{gather}
     \bar{g}_{b,m}=\left[
     \begin{array}{cc}
      1 & 0 \\
      b & M \\
     \end{array}
     \right].
\end{gather}
As it turns out, the matrices $\bar{g}_{b,M}$, for
$b\in\QQ\cap[0,1)$ and $M\in\QQp$, also furnish a complete and
irredundant list of coset representatives for $\PSL_2(\ZZ)$ in
$\PSL_2(\QQ)$. We will write $\bar{L}_{b,M}$ for the projective
lattice corresponding to $\PSL_2(\ZZ)\bar{g}_{b,M}$. Of course,
every projective lattice $L_{M,b}$ can be written as
$\bar{L}_{b',M'}$ for some $b',M'$. The correspondence is such that
\begin{gather}
     L_{M,0}=\bar{L}_{0,1/M},\quad
     L_{M,f/g}=\bar{L}_{f'/g,1/g^2M},
\end{gather}
where, for $0< f<g$ and $\gcd\{f,g\}=1$, the integer $f'$ is
uniquely determined by the conditions that $0<f'<g$ and $ff'\equiv
1\pmod{g}$.

When $L_{M,b}=\bar{L}_{b',M'}$ we (following
\cite{Con_UndrstndgGamma0N}) call $\bar{L}_{b',M'}$ the {\em reverse
name} for $L_{M,b}$. The reverse names for projective lattices will
be useful in \S\ref{sec:arith:orbits}.

\subsection{Stabilizers}\label{sec:arith:stabs}

Let us continue with the notation and conventions of
\S\ref{sec:arith:names}. In particular, we consider the
$\QQm$-orbits of lattices in $V$ that are commensurable with our
distinguished lattice $L_1$ --- the set $P\mc{L}_1$, of projective
lattices in $V_1$
--- and these orbits are in natural correspondence with the
$\PSL_2(\ZZ)\backslash\PGL_2^+(\QQ)$. Each projective lattice $L\in
P\mc{L}_1$ arises as $L=L_{M,b}$ for some $M\in \QQp$ and
$b\in\QQ\cap[0,1)$.

From now on, let us write $G$ for the group $\PGL_2^+(\QQ)$ regarded
as a group with right-action on the space $P\mc{L}_1$ of projective
lattices in $V_1$.
\begin{gather}
     \begin{split}
     P\mc{L}_1&\leftrightarrow
          \PSL_2(\ZZ)\backslash\PGL_2^+(\QQ)
               \circlearrowleft
               \PGL_2^+(\QQ)=:G\\
          \PSL_2(\ZZ)\cdot(g\cdot[{\bf v_1}])&\leftrightarrow
               \PSL_2(\ZZ)g
     \end{split}
\end{gather}

For any $L\in P\mc{L}_1$, we may ask for the subgroup of $G$ that
fixes $L$; i.e. the group $\Fix_G(L)$. We will write $G_L$ for
$\Fix_G(L)$, and $G_{M,b}$ for $G_L$ if $L=L_{M,b}$. We may also
consider the group that fixes several lattices,
\begin{gather}
     G_{(L^1,\ldots,L^k)}=\bigcap_iG_{L^i}=
          \left\{g\in G\mid L^i\cdot
     g=L^i,\;1\leq i\leq k\right\},
\end{gather}
or the group that stabilizes a set of lattices,
\begin{gather}
     G_{\{L^1,\ldots,L^k\}}=\left\{g\in G\mid \{L^1\cdot
     g,\ldots,L^k\cdot g\}=\{L^1,\ldots, L^k\}\right\}.
\end{gather}
Let us compute some examples of the groups $G_{(L,\cdots)}$,
$G_{\{L,\cdots\}}$. Note that, by definition,
\begin{gather}\label{eqn:arith:stabs|ractn}
     (\PSL_2(\ZZ)\cdot[{\bf v}_1])\cdot g=
     \PSL_2(\ZZ)g\cdot [{\bf v}_1]
\end{gather}
for $g\in G$. The identity (\ref{eqn:arith:stabs|ractn}) makes it
clear that the group $G_1=\Fix_G(L_1)$ is none other than the
familiar {\em modular group}, $\PSL_2(\ZZ)$. Evidently, $G$ acts
transitively on $P\mc{L}_1$, so the lattice $L_M$, for $M\in\QQp$,
is fixed by a conjugate of $G_1\cong \PSL_2(\ZZ)$; viz.
\begin{gather}
     G_M=g_M^{-1}
     G_1
     g_M
     =\left\{
     \left[
       \begin{array}{cc}
         a & b/M \\
         cM & d \\
       \end{array}
     \right]
     \mid a,b,c,d\in\ZZ,\; ad-bc= 1
     \right\},
\end{gather}
Thus we see that the intersection $G_{(1,N)}=G_1\cap G_N$, for $N$ a
positive integer, is the {\em Hecke group of level $N$}, usually
denoted $\Gamma_0(N)$.

For $N$ a prime power, $N=p^a$ say, the group $G_{\{1,N\}}$ that
preserves the set $\{L_1,L_N\}$ is the group obtained from
$G_{(1,N)}$ by adjoining a {\em Fricke involution}:
\begin{gather}
     G_{ \{1,N\} }=
     \left
     \langle
     G_{(1,N)},\left[
               \begin{array}{cc}
                 0 & -1 \\
                 N & 0 \\
               \end{array}
             \right]
     \right
     \rangle.
\end{gather}
This is the group denoted $\Gamma_0(N)+N$ (or just $\Gamma_0(N)+$,
since $N$ is a prime power) in \cite{ConNorMM}. Notice that the
subgroup of $G$ stabilizing the set $\{L_1,L_{p^a}\}$ is the same as
the subgroup of $G$ stabilizing the set
$\{L_1,L_p,\ldots,L_{p^a}\}$, containing the lattices associated to
all powers of $p$ that divide $p^a$, and this latter set is
precisely the set of lattices $L$ for which the product of
hyperdistances $\delta(L_1,L)\delta(L,L_{p^a})$ coincides with the
hyperdistance between $L_1$ and $L_{p^a}$. Given $L'$ and $L''$ in
$P\mc{L}_1$, let us write $(L',L'')+$ for the set of lattices $L$
satisfying $\delta(L',L)\delta(L,L'')=\delta(L',L'')$.
\begin{gather}
     (L',L'')+
     =\left\{
     L\in P\mc{L}_1\mid
     \delta(L',L)\delta(L,L'')=\delta(L',L'')
     \right\}
\end{gather}
Following \cite{Con_UndrstndgGamma0N}, we call the set $(L',L'')+$
the {\em $(L',L'')$-thread}. The sets $(1,6)+$ and $(1,4)+$ appear
in (\ref{diag:arith:nodes|threads}).
\begin{equation}\label{diag:arith:nodes|threads}
     \xy
     (-25,0)*+{2}="2";
     (0,0)*+{6}="6";
     (-25,-20)*+{1}="1";
     (0,-20)*+{3}="3";
     {\ar@{-} "1"; "2"};
     {\ar@{-} "3"; "6"};
     {\ar@{=} "1"; "3"};
     {\ar@{=} "2"; "6"};
     \endxy
     \qquad
     \xy
     (-18,-15)*+{1}="1";
     (0,-5)*+{2}="2";
     (18,-15)*+{4}="4";
     {\ar@{-} "1"; "2"};
     {\ar@{-} "2"; "4"};
     \endxy
\end{equation}
In terms of threads, we have $G_{\{1,N\}}=G_{(1,N)+}$ when $N$ is a
prime power. More generally, for arbitrary $N\in \ZZp$, the group
$G_{(1,N)+}$ is the group obtained by adjoining to $G_{(1,N)}$, the
sets $W_e=W_e(N)$, where $e\in\ZZp$ is a divisor of $N$, and
\begin{gather}\label{eqn:arith:stabs|alinv}
     W_e=\left\{
     \left[
       \begin{array}{cc}
       ae & b \\
         cN & de \\
       \end{array}
     \right]
     \mid
     a,b,c,d\in\ZZ,\;ade^2-bcN=e
     \right\}.
\end{gather}
Note that $ade^2-bcN=e$ implies $\gcd\{e,N/e\}=1$, so the set $W_e$
is non-empty only when $e$ is an {\em exact divisor} of $N$. The
sets $W_e$ are exactly the cosets of $G_{(1,N)}=W_1$ in
$G_{(1,N)+}$. (We will sometimes write $W_e$ for a certain element
in $W_e(N)$ --- the validity of the statement in which this $W_e$
appears should be independent of the choice that is made.)

Suppose given $h,n\in\ZZp$ such that $h|n$. Then
\begin{gather}
     G_{(h,n)}=\left\{
          \left[
            \begin{array}{cc}
              a & b/h \\
              cn & d \\
            \end{array}
          \right]
          \mid
          a,b,c,d\in\ZZ,\;
          ad-bcn/h= 1\right\}.
\end{gather}
This group is conjugate to $G_{(1,n/h)}$ since
$G_{(h,n)}=g_h^{-1}G_{(1,n/h)}g_h$ (c.f.
(\ref{eqn:arith:names|gs})). The group $G_{(h,n)+}$ is evidently
obtained from $G_{(1,n/h)+}$ via conjugation by $g_h$, and so
$G_{(h,n)+}$ consists of $G_{(h,n)}$ together with the (other
co)sets $g_h^{-1}W_e(n/h)g_h$ for $e$ an exact divisor of $n/h$. 
The significance of the groups $G_{(h,n)+}$ is demonstrated by the
following result.
\begin{thm}[Atkin--Lehner]\label{thm:arith:stabs|AL}
For $N$ a positive integer, the normalizer of $G_{(1,N)}$ in
$\PSL_2(\RR)$ is the group $G_{(h,n)+}$, where $n=N/h$, and $h$ is
the largest divisor of $24$ such that $h^2|N$.
\end{thm}
\begin{rmk}
A beautiful proof of the Atkin--Lehner Theorem appears in
\cite{Con_UndrstndgGamma0N}.
\end{rmk}

Recall that the action of $G$ on $P\mc{L}_1$ preserves hyperdistance
(c.f. \S\ref{sec:arith:hdist}). It follows that for any $L\in
P\mc{L}_1$, the group $G_L$ acts by permutations on the set
$HC_{N}(L)$, of lattices at hyperdistance $N$ from $L$. We call
$HC_N(L)$ the {\em hypercircle of hyperradius $N$ about $L$}.
\begin{gather}
     HC_N(L)=\left\{
          L'\in P\mc{L}_1\mid \delta(L,L')=N\right\}
\end{gather}
In fact, the action of $G_L$ on $HC_N(L)$ is transitive.
\begin{prop}\label{prop:arith:stabs|GLtrans}
For any positive integer $N$, the group $G_L=\Fix_G(L)$ acts
transitively on the set of lattices that are hyperdistant $N$ from
$L$.
\end{prop}
We will furnish a proof of Proposition
\ref{prop:arith:stabs|GLtrans} in \S\ref{sec:arith:orbits}.

We may consider the subgroup of $G$ that fixes all the lattices in
$HC_N(L)$ for given $N\in\ZZp$ and $L\in P\mc{L}_1$. When $L$ is the
distinguished lattice $L=L_1$, this is just the group $\Gamma(N)$
--- the {\em principal congruence group of level $N$}.
\begin{gather}
     1\to\Gamma(N)\to\PSL_2(\ZZ)\to\PSL_2(\ZZ/N)\to 1
\end{gather}
Recall that a group $H<G_1\cong\PSL_2(\ZZ)$ is called a {\em
congruence group} if $H$ contains $\Gamma(N)$ for some $N$. A
congruence group $H$ is said to have {\em level $N$} if $N$ is the
smallest positive integer such that $\Gamma(N)<H$. We will apply
these definitions also to arithmetic groups: we will say that a
subgroup $\Gamma<\PSL_2(\RR)$ that is commensurable with $G_1$ is a
{\em congruence group} if $\Gamma$ contains $\Gamma(N)$ for some
positive integer $N$, and we will say that such a group $\Gamma$ has
{\em level $N$}, if $N$ is the minimal positive integer for which
$\Gamma(N)<\Gamma$.

\begin{rmk}
See Lemma \ref{lem:arith:cusps|HCppow} for an explicit description
of $HC_N(1)$ in the case that $N$ is a prime power.
\end{rmk}

\subsection{Trees}\label{sec:arith:trees}

Let $p$ be a prime, and consider the set $(P\mc{L}_1)_p\subset
P\mc{L}_1$ consisting of (projective) lattices that are hyperdistant
a power of $p$ from $L_1$.
\begin{gather}
     (P\mc{L}_1)_p=\left\{L\in P\mc{L}_1\mid \delta(L_1,L)\in
     p^{\ZZ}
          \right\}
\end{gather}
Following Conway \cite{Con_UndrstndgGamma0N}, we call
$(P\mc{L}_1)_p$ the {\em $p$-adic tree} in $P\mc{L}_1$.

For $L,L'\in P\mc{L}_1$ we say $L$ and $L'$ are {\em $p$-adically
equivalent}, and write $L\sim_p L'$, if $p\nmid\delta(L,L')$. It is
easy to check that the $p$-adic equivalence class of any projective
lattice $L\in P\mc{L}_1$ has a unique representative in
$(P\mc{L}_1)_p$.
\begin{prop}\label{prop:arith:trees|uniquepadicrep}
For any $L\in P\mc{L}_1$, and any prime $p$, there is a unique
$L'\in P\mc{L}_1$ such that $L\sim_pL'$ and $L'\in(P\mc{L}_1)_p$.
\end{prop}
Consequently, we obtain a map $\pi_p:P\mc{L}_1\to (P\mc{L}_1)_p$,
called the {\em $p$-adic projection}, by setting $\pi_p(L)=L'$ when
$L$ and $L'$ are as in Proposition
\ref{prop:arith:trees|uniquepadicrep}.

As explained in \cite{Con_UndrstndgGamma0N}, the set $(P\mc{L}_1)_p$
has a natural tree structure; viz. if we regard the elements of
$(P\mc{L}_1)_p$ as vertices of a graph, with edges joining just
those lattices that are hyperdistant $p$ from each other, we obtain
a graph with the property that there is a unique shortest path
between any two vertices. That is, we obtain a tree. The $p$-adic
tree $(P\mc{L}_1)_p$ has infinitely many nodes, and each node is
$p+1$ valent.

A finite subset $S\subset P\mc{L}_1$ is called a {\em cell} if, for
each prime $p$, the subtree of $(P\mc{L}_1)_p$ generated by the set
$\pi_p(S)\subset (P\mc{L}_1)_p$ is either a point, or two points
joined by an edge.

The methods of \cite{Con_UndrstndgGamma0N} illustrate the utility of
the trees $(P\mc{L}_1)_p$ and the projections $\pi_p$. The following
result is a prime example of this.
\begin{prop}[\cite{Con_UndrstndgGamma0N}]\label{prop:arith:trees|stabcell}
If $\Gamma$ is a subgroup of $\PGL_2^+(\QQ)$ that is commensurable
with $G_1\cong\PSL_2(\ZZ)$, then $\Gamma$ stabilizes a cell in
$P\mc{L}_1$.
\end{prop}
It is well known that a subgroup of $\PGL_2^+(\RR)$ that is
commensurable with $\PSL_2(\ZZ)$ is contained in $\PGL_2^+(\QQ)$
(c.f. Proposition \ref{prop:arith:cusps|commenswG1inQ}). For any
cell $S\subset P\mc{L}_1$ we can find lattices $L$ and $L'$ in $S$
that maximize $\delta(L,L')$. After conjugation by an element of
$\PGL_2^+(\QQ)$ we may assume that $L=L_1$. Since $G_1\cong
\PSL_2(\ZZ)$ acts transitively on the lattices of any given
hyperdistance from $L_1$ (c.f. Proposition
\ref{prop:arith:stabs|GLtrans}), we may conjugate the pair $(L,L')$
to $(L_1,L_N)$ for some $N$. These observations, together with
Proposition \ref{prop:arith:trees|stabcell}, quickly imply the
following result, known as Helling's Theorem.
\begin{thm}[\cite{Hel_GpsCommensModGp}]
The maximal arithmetic subgroups of $\PSL_2(\RR)$ are the conjugates
of $G_{(1,N)+}$ for square-free $N$.
\end{thm}

\subsection{Characters}\label{sec:arith:chars}

Let $N$ be a positive integer, and let $h$ be the largest divisor of
$24$ such that $h^2|N$. Observe that $G_{(h,N/h)}$ contains
$G_{(1,N)}$. According to \cite{ConNorMM} (c.f.
\cite{ConMcKSebDiscGpsM}), there is a canonically defined subgroup
of index $h$ in $G_{(h,N/h)}$ that contains $G_{(1,N)}$, and it may
be realized as the kernel of a (non-canonically defined)
homomorphism
\begin{gather}
     \lambda:G_{(h,N/h)}\to \ZZ/h.
\end{gather}
We will write $G_{(h,N/h)}^{(h)}$ for this group that arises as
$\ker(\lambda)$. This group is denoted $\Gamma_0(n|h)$ in
\cite{ConNorMM}.

The cases that $N=9$ and $N=8$ will be of particular relevance in
\S\ref{sec:dyn}.

\subsubsection{Example: $N=9$.}\label{sec:arith:chars:N=9}

If $N=9$ then $h=3$ and $N/h=3$. Evidently, $G_{(3,3)}=G_3$. Define
generators $x$ and $y$ for $G_3/G_{(1,9)}$ (recall from Theorem
\ref{thm:arith:stabs|AL} that $G_{(1,9)}$ is normal in $G_3$) by
setting
\begin{gather}
     x=G_{(1,9)}T^{1/3}
     =G_{(1,9)}\left[
         \begin{array}{cc}
           1 & 1/3 \\
           0 & 1 \\
         \end{array}
       \right],\\
       y=G_{(1,9)}(T^3)^t
       =G_{(1,9)}\left[
           \begin{array}{cc}
             1 & 0 \\
             3 & 1 \\
           \end{array}
         \right].
\end{gather}
A suitable map $\lambda: G_3/G_{(1,9)}\to\ZZ/3$ may be defined by
assigning $\lambda(y)=\sigma$ and $\lambda(x)=\sigma^{-1}$, where
$\sigma$ is a generator for $\ZZ/3$ (c.f. \cite{ConMcKSebDiscGpsM}).

Let us also analyze the situation in terms of the $3$-adic tree
(c.f. \S\ref{sec:arith:trees}). The group $G_3$ preserves the set
$HC_3(3)$ --- the hypercircle of hyperradius $3$ about $L_3$. There
are exactly $4$ lattices in $HC_3(3)$; they appear in the diagram
(\ref{diag:dyn:nodes|HC3(3)}), which displays the smallest subtree
of the $3$-adic tree $(P\mc{L}_1)_3$ that contains $HC_3(3)$.
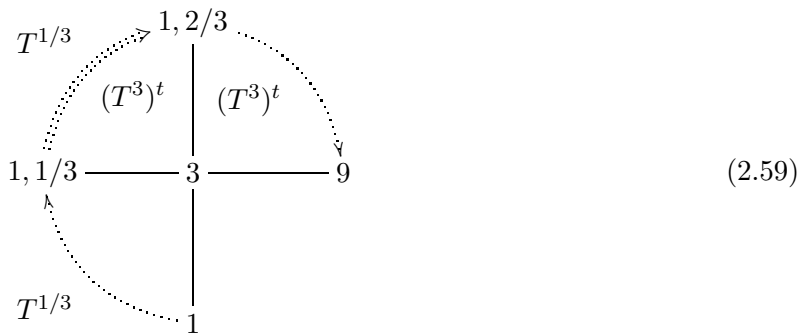
\begin{equation}\label{diag:dyn:nodes|HC3(3)}
     \xy
     (0,20)*+{{1,{2}/{3}}}="12t";
     (0,0)*+{3}="3";
     (-20,0)*+{1,1/3}="11t";
     (20,0)*+{9}="9";
     (0,-20)*+{1}="1";
     {\ar@{-} "12t";"3"};
     {\ar@{-} "11t";"3"};
     {\ar@{-} "1";"3"};
     {\ar@{-} "9";"3"};
     {\ar^*+{T^{1/3}} @/^1.3pc/ @{.>} "1";"11t"};
     {\ar^*+{T^{1/3}} @/^1.3pc/ @{:>} "11t";"12t"};
     (-8,10)*+{(T^3)^t};
     {\ar_*+{(T^3)^t} @/^1.3pc/ @{.>} "12t";"9"};
     \endxy
\end{equation}
Observe that $G_{1,9}$ fixes every element of $HC_3(3)$. Thus, there
is a natural map $G_3/G_{1,9}\to{\rm Sym}(HC_3(3))$. Diagram
(\ref{diag:dyn:nodes|HC3(3)}) also displays the $3$-cycles in
$\Sym(HC_3(3))$ generated by (the images in $\Sym(HC_3(3))$ of) $x$
and $y$. Evidently, the image of $G_3/G_{1,9}$ in ${\rm
Sym}(HC_3(3))$ is just ${\rm Alt}(HC_3(3))$ --- a copy of the
alternating group on $4$ letters. This group has order $12$. It
contains $8$ elements of order $3$, and $3$ elements of order $2$.
Any non-trivial homomorphism ${\rm Alt}(HC_3(3))\to\ZZ/3$ must be
trivial on elements of order $2$, and non-trivial on the elements of
order $3$.

We conclude that $G_3^{(3)}$ consists of all the elements of $G_3$
that induce permutations of order $2$ (or $1$) on the set $HC_3(3)$,
and the permutations of order $2$ occurring are just those that
arise as a product of $2$ disjoint transpositions.
\begin{gather}
     G_{3}^{(3)}=\lab
          G_{(1,9)},(T^3)^tT^{1/3},T^{-1/3}(T^3)^tT^{-1/3}
               \rab
\end{gather}
The group $G_{3}^{(3)}$ is denoted $\Gamma_0(3|3)$ in
\cite{ConNorMM}.

\subsubsection{Example: $N=8$.}\label{sec:arith:chars:N=8}

If $N=8$ then $h=2$ and $N/h=4$. The group $G_{(2,4)}$ contains
$G_{(1,8)}$ normally by Theorem \ref{thm:arith:stabs|AL}. Consider
the set $S=HC_2(2)\cup HC_2(4)$.
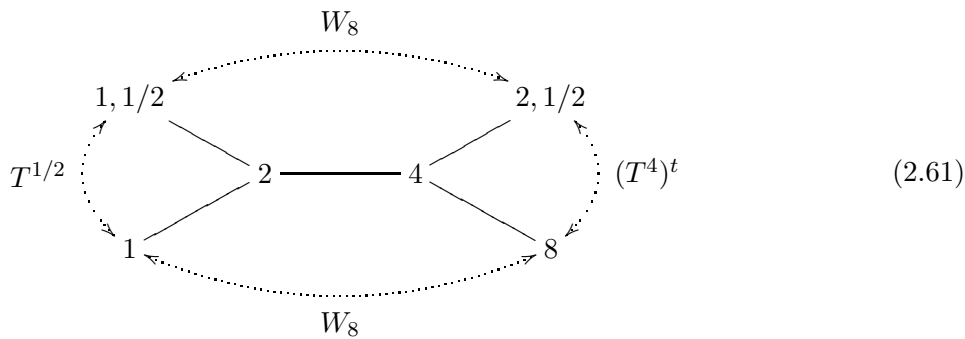
\begin{equation}\label{diag:dyn:nodes|HC2(2)cupHC2(4)}
     \xy
     (38,10)*+{2,1/2}="2h";
     (20,0)*+{4}="4";
     (38,-10)*+{8}="8";
     (0,0)*+{2}="2";
     (-18,-10)*+{1}="1";
     (-18,10)*+{1,1/2}="1h";
     {\ar@{-} "2h"; "4"};
     {\ar@{-} "2"; "4"};
     {\ar@{-} "8"; "4"};
     {\ar@{-} "1h"; "2"};
     {\ar@{-} "1"; "2"};
     {\ar^*+{(T^4)^t} @/^1.5pc/ @{<.>} "2h";"8"};
     {\ar^*+{W_8} @/^1.5pc/ @{<.>} "1h";"2h"};
     {\ar_*+{W_{8}} @/_1.5pc/ @{<.>} "1";"8"};
     {\ar^*+{T^{1/2}} @/^1.5pc/ @{<.>} "1"; "1h"};
     \endxy
\end{equation}
Diagram (\ref{diag:dyn:nodes|HC2(2)cupHC2(4)}) displays the smallest
subtree of the $2$-adic tree (c.f. \S\ref{sec:arith:trees})
containing $S$, together with the actions on this tree induced by
some elements of $G\cong\PGL_2^+(\QQ)$ (c.f.
\S\ref{sec:arith:stabs}). The symbol $W_8$ in
(\ref{diag:dyn:nodes|HC2(2)cupHC2(4)}) denotes an (arbitrary)
element of $g_2^{-1}W_2(2)g_2$ (c.f. \S\ref{sec:arith:stabs}).

The group $G_{(2,4)}$ acts by permutations on $S$; the subgroup
$G_{(1,8)}$ is exactly the subgroup that fixes every element of $S$.
We see from (\ref{diag:dyn:nodes|HC2(2)cupHC2(4)}) that $G_{(2,4)}$
is generated by $G_{(1,8)}$ together with $T^{1/2}$ and $(T^4)^t$.
The group $G_{\{2,4\}}$ --- the full normalizer of $G_{(1,8)}$ ---
is obtained by adjoining $W_8$ to $G_{(2,4)}$.

We thus obtain generators $x,y$ for $G_{(2,4)}/G_{(1,8)}$ by setting
$x=G_{(1,8)}T^{1/2}$ and $y=G_{(1,8)}(T^4)^t$, say. We obtain a
suitable map $\lambda:G_{(2,4)}/G_{(1,8)}\to \ZZ/2$ by setting
$\lambda(x)=\lambda(y)=\sigma$, where $\sigma$ is the non-trivial
element of $\ZZ/2$.
\begin{gather}\label{eqn:arith:chars:N=8|G24}
     G_{(2,4)}^{(2)}=\lab
          G_{(1,8)},(T^4)^tT^{1/2}
               \rab
\end{gather}
We can also carry out this story with $G_{\{2,4\}}$ in place of
$G_{(2,4)}$; that is, we can adjoin the set $g_2^{-1}W_2(2)g_2$ to
(\ref{eqn:arith:chars:N=8|G24}), just as we adjoin this same set to
$G_{(2,4)}$ in order to recover $G_{(2,4)}$ (c.f.
(\ref{eqn:arith:stabs|alinv})). We set
\begin{gather}\label{eqn:arith:chars:N=8|G24plus}
     G_{\{2,4\}}^{(2)}=\lab
          G_{(1,8)},(T^4)^tT^{1/2},W_8
               \rab.
\end{gather}
Evidently, $G_{\{2,4\}}^{(2)}$ may be characterized as the kernel of
the following composition of natural mappings:
$G_{\{2,4\}}\to\Sym(S)\to\ZZ/2$. This group is denoted
$\Gamma_0(4|2)+$ in \cite{ConNorMM}.

\subsection{Cusps}\label{sec:arith:cusps}

From now on we will restrict to the case that $\kk=\RR$.

There is an obvious embedding $\PSL_2(\RR)\hookrightarrow
\PGL_2^+(\RR)$ arising from the embedding of groups
$\SL_2(\RR)\hookrightarrow \GL_2^+(\RR)$. Observe that this map
\begin{gather}
     \PSL_2(\RR)\hookrightarrow \PGL_2^+(\RR)
\end{gather}
is in fact surjective, so that $\PGL_2^+(\RR)=\PSL_2(\RR)$. There is
also an embedding $\GL_2^+(\QQ)\hookrightarrow\GL_2^+(\RR)$ coming
the the embedding of fields $\QQ\hookrightarrow \RR$. The map
$\PGL_2^+(\QQ)\to\PGL_2^+(\RR)$ sending a $\QQm$-orbit in
$\GL_2^+(\QQ)$ to the $\RRm$-orbit of its image in $\GL_2^+(\RR)$ is
readily checked to be an embedding; we conclude that $\PGL_2^+(\QQ)$
embeds naturally in $\PSL_2(\RR)$.

Our choice ${\bf v}_1\in\mc{B}^+_1$ allows us to identify
$PV^{\times}=\RRm\backslash V^{\times}$ with the real projective
line $\PP^1(\RR)=\RR\cup\{\infty\}$; viz.
\begin{gather}
     [\alpha u_1+ \beta v_1]
     \leftrightarrow
     \begin{cases}
          \alpha/\beta,&\text{ if $\beta\neq 0$;}\\
     \infty,&\text{ if $\beta=0$;}
     \end{cases}
\end{gather}
where ${\bf v}_1=(u_1,v_1)$. The resulting left action of
$\PGL_2(\RR)$ on $PV^{\times}$ is given by
\begin{gather}
     \left[
       \begin{array}{cc}
         a & b \\
         c & d \\
       \end{array}
     \right]\cdot [\alpha u_1+\beta v_1]
     =[(a\alpha +b\beta )u_1+(c\alpha+d\beta)v_1].
\end{gather}
In terms of the identification $PV^{\times}\leftrightarrow
\RR\cup\{\infty\}$ this translates to the familiar prescription
\begin{gather}
     \left[
       \begin{array}{cc}
         a & b \\
         c & d \\
       \end{array}
     \right]\cdot \alpha
     =\frac{a\alpha +b }{c\alpha +d }.
\end{gather}
Evidently, it is natural to regard the group $\PGL_2^+(\RR)$ as
acting both from the right on projective lattices $P\mc{L}$, and
from the left on the projective line $\PP^1(\RR)$. Similarly, we may
regard our group $G\cong\PGL_2^+(\QQ)$ (see \S\ref{sec:arith:stabs})
as acting both from the right on the projective lattices $P\mc{L}_1$
in $V_1$, and from the left on the rational projective line
$\PP^1(\QQ)$. Also, there is a natural embedding
$\PP^1(\QQ)\hookrightarrow \PP^1(\RR)$.

Recall that $T$ typically denotes the element of
$\PSL_2(\RR)=\PGL_2^+(\RR)$ represented by the upper-triangular
unipotent matrix with $1$ in the top right-hand corner. We convene
to write $T^A$, for $A\in\RR$, for the element of $\PGL_2^+(\RR)$
represented by the upper-triangular unipotent matrix with $A$ in the
top right-hand corner.
\begin{gather}
     T^A:=\left[
            \begin{array}{cc}
              1 & A \\
              0 & 1 \\
            \end{array}
          \right]
          \in \PGL_2^+(\RR)
\end{gather}
Evidently, $T^AT^B=T^{A+B}$ for $A,B\in\RR$, so that the assignment
$A\mapsto T^A$ defines an embedding (of groups) $\RR\hookrightarrow
\PGL_2^+(\RR)$. Observe that the stabilizer of $\infty\in\PP^1(\RR)$
in $\PGL_2^+(\RR)$ (let us write $\Fix(\infty)$ for this group)
contains the image of $\RR$ under this embedding.
\begin{gather}\label{eqn:arith:cusps|fixinfty}
     \Fix(\infty)=\left\{\left[
       \begin{array}{cc}
         a & b \\
         0 & 1 \\
       \end{array}
     \right]\mid a,b\in\RR,\;a>0\right\}
\end{gather}
In fact, the image of $\RR$ here is a normal subgroup of
$\Fix(\infty)$, and $\Fix(\infty)$ may be described as the split
extension $\Fix(\infty)\cong \RR\rtimes\RRp$, corresponding to the
natural action of (the multiplicative group) $\RRp$ on (the additive
group) $\RR$. Given a subgroup $H<\PGL_2^+(\RR)$ let us write
$\Fix_H(\infty)$ for the intersection $\Fix(\infty)\cap H$. We have
\begin{gather}\label{eqn:arith:cusps|FixGinfty}
     \Fix_G(\infty)\cong \QQ\rtimes\QQp.
\end{gather}
Observe that $\Fix_G(\infty)$ acts transitively on the subset
$\PP^1(\QQ)\subset\PP^1(\RR)$, and no element of $P\mc{L}_1$ is
fixed by any non-trivial element of the subgroup
$\QQp<\Fix_G(\infty)$. Observe also that the intersection
$\Fix_{G_1}(\infty)=\Fix_G(\infty)\cap G_1$ is contained in the
image (under $A\mapsto T^A$) of $\QQ$; indeed,
$\Fix_{G_1}(\infty)=\lab T\rab$. This property is shared by the
subgroups of $\PSL_2(\RR)$ that are commensurable with
$G_1\cong\PSL_2(\ZZ)$ (c.f. \S\ref{sec:arith:stabs}).
\begin{prop}\label{prop:arith:cusps|fixinftycommensG1inQQ}
Suppose $H<\PSL_2(\RR)$ is commensurable with $G_1$. Then
$\Fix_H(\infty)$ is contained in the image of $\QQ$ under $A\mapsto
T^A$.
\end{prop}
\begin{proof}
Let $h\in \Fix_H(\infty)$. Choose $a,b\in \RR$ such that
\begin{gather}
     h=\left[
       \begin{array}{cc}
         a & b \\
         0 & 1 \\
       \end{array}
     \right].
\end{gather}
If $H$ is commensurable with $G_1$ then $h^n\in G_1$ for some $n\in
\ZZ$. This implies that $a^n=1$, and hence $a=1$, since $a$ must be
positive. Given that $a=1$, we have $nb\in \ZZ$, so that $b\in \QQ$,
and $h=T^b$ lies in the image of $\QQ$ in $\Fix_G(\infty)$, as
required.
\end{proof}

Recall that an element $\gamma\in\PSL_2(\RR)$ is called {\em
parabolic} if there is a unique fixed point for its action on
$\PP^1(\RR)$. Equivalently, $\gamma\in\PSL_2(\RR)$ is parabolic if
it is not the identity, and if $\gamma=[A]$ for some
$A\in\SL_2(\RR)$ with ${\rm tr}(A)=2$. For any subgroup
$H<\PGL_2^+(\RR)$ we may ask for the orbits of $H$ on $\PP^1(\RR)$
that are comprised of points that are fixed by some parabolic
element of $H$. (If $h\cdot s=s$ for some parabolic $h\in H$ and
some $s\in\PP^1(\RR)$ then every $s'\in H\cdot s$ is fixed by some
parabolic element of $H$.) Such an orbit for the action of $H$ on
$\PP^1(\RR)$ is called a {\em cusp} of $H$. We write $\mc{C}(H)$ for
the set of cusps of $H$, and we define
\begin{gather}
     C(H):=\bigcup_{H\cdot s\in\mc{C}(H)}H\cdot s,
\end{gather}
so that $C(H)$ consists of all the points in $\PP^1(\RR)$ which are
fixed by some parabolic element of $H$.
\begin{prop}\label{prop:arith:cusps|commens}
The assignment $H\mapsto C(H)$ is constant on commensurability
classes.
\end{prop}
\begin{proof}
It suffices to show that $C(H)=C(H')$ whenever $H'$ is a subgroup of
finite index in $H$. Certainly, $C(H')\subset C(H)$ in this case, so
let $s\in C(H)$. Then there is some parabolic $h\in H$ such that
$h\cdot s=s$. Since $H'$ has finite index in $H$, there is some
$n>0$ such that $h^n\in H'$, but $h^n$ is parabolic whenever $h$ is
(since $A^2={\rm tr}(A)A-1$ for $A\in\SL_2(\RR)$), so there is a
parabolic element of $H'$ that fixes $s$, and $C(H)\subset C(H')$,
as required.
\end{proof}
For $H<\PGL_2^+(\RR)$ set $c(H)$ to be the number of cusps of $H$.
\begin{gather}
     c(H):=|\mc{C}(H)|
\end{gather}
It is easy to check that $C(G)=\PP^1(\QQ)$, and it is clear that $G$
acts transitively on $\PP^1(\QQ)$, so we have $c(G)=1$. In fact,
these statements remain valid with $G_1\cong\PSL_2(\ZZ)$ in place of
$G$.
\begin{gather}
     c(G_1)=1,\quad
     C(G_1)=\PP^1(\QQ),\quad
     \mc{C}(G_1)=\{G_1\cdot \infty\}.
\end{gather}
By Proposition \ref{prop:arith:cusps|commens}, we have
$C(H)=\PP^1(\QQ)$ for any subgroup of $\PGL_2^+(\RR)$ that is
commensurable with $G_1$. This has the following useful consequence.
\begin{prop}\label{prop:arith:cusps|commenswG1inQ}
If $H<\PSL_2(\RR)$ is commensurable with $G_1\cong\PSL_2(\ZZ)$ then
$H$ is contained in $G\cong\PGL_2^+(\QQ)$.
\end{prop}
\begin{proof}
We see from Proposition \ref{prop:arith:cusps|commens} that the
action of such a group $H$ on $\PP^1(\RR)$ stabilizes $\PP^1(\QQ)$.
Any element of $H$ can be written as a M\"obius transformation, and
in particular, is determined by its action on $0$, $1$ and $\infty$,
for example. Since these points are mapped to points of
$\PP^1(\QQ)$, any element of $H$ can be represented by a rational
matrix.
\end{proof}
Let $H<G$ such that $C(H)=\PP^1(\QQ)$. Then $c(H^g)=c(H)$ for any
$g\in G$, since $G$ acts transitively on $\PP^1(\QQ)$. This shows,
for example, that $c(G_L)=1$ for any $L\in P\mc{L}_1$.

\medskip

We may ask what the action of a group on $P\mc{L}_1$ says about its
cusps. The following lemma is useful in this regard.
\begin{lem}\label{lem:arith:cusps|bij}
Let $H$ be a group equipped with a transitive left-action on a set
$L$, and a transitive right-action on a set $R$. Let $l\in L$ and
$r\in R$, and write $H_x$ for the stabilizer of $x$ in $H$, for $x$
in $L$ or $R$. Then the $H_{r}$-orbits on $L$ are in bijective
correspondence with the $H_{l}$-orbits on $R$.
\begin{gather}\label{eqn:arith:cusps|bij}
     L/H_{r}\leftrightarrow H_{l}\backslash R
\end{gather}
\end{lem}
\begin{proof}
Since the $H$-actions are transitive, we have bijections
$L\leftrightarrow H_l\backslash H$ and $H/H_r\leftrightarrow R$, so
both sides of (\ref{eqn:arith:cusps|bij}) are in bijection with
$H_l\backslash H/H_r$. \end{proof} Note that the correspondence in
(\ref{eqn:arith:cusps|bij}) is given explicitly by
\begin{gather}
     l\cdot hH_r\leftrightarrow H_lh\cdot r.
\end{gather}
Lemma \ref{lem:arith:cusps|bij} has the following immediate
consequence.
\begin{prop}\label{prop:arith:cusps|cusps}
Let $H$ be a subgroup of $G$ with a single cusp, let $\mc{O}$ be an orbit for the action of $H$ on $P\mc{L}_1$, and let $H_0$ be the subgroup of $H$ fixing some $L_0\in \mc{O}$. Then the cusps of $H_0$ are in 
bijective correspondence with the orbits of $\Fix_H(\infty)$ on
$\mc{O}$.
\end{prop}
Suppose now that $H$ is a subgroup of $\PSL_2(\RR)$ commensurable
with $G_1$, so that $H$ is contained in $G\cong\PGL_2^+(\QQ)$, by
Proposition \ref{prop:arith:cusps|commenswG1inQ}, and suppose that
$H$ has a single cusp (i.e. acts transitively on $\PP^1(\QQ)$). For
a subgroup $H_0$ of $H$, we have $\mc{C}(H_0)=H_0\backslash
\PP^1(\QQ)$ for the cusps of $H_0$. If $H_0$ has finite index in $H$
then we may define a function
\begin{gather}\label{eqn:arith:cusps|widthfn}
     w_H:\mc{C}(H_0)\to\QQp
\end{gather}
in the following way. For $H_0\cdot x\in H_0\backslash\PP^1(\QQ)$
let $g\in H$ such that $x=g\cdot \infty$, and consider the
intersection $H_0^g\cap\Fix_H(\infty)$. Since $H$ is commensurable
with $G_1$, the group $\Fix_H(\infty)$ is an infinite cyclic group,
by Proposition \ref{prop:arith:cusps|fixinftycommensG1inQQ}. Since
$H_0$ has finite index in $H$, the intersection $H_0^g\cap
\Fix_H(\infty)$ is also an infinite cyclic group. Define
$w_H(H_0\cdot x)=A\in\QQp$ just when $T^A$ generates
$H_0^g\cap\Fix_H(\infty)$.
\begin{prop}
The function $w_H$, of (\ref{eqn:arith:cusps|widthfn}), is
well-defined.
\end{prop}
\begin{proof}
If $g,g'\in H$ satisfy $g\cdot\infty=g'\cdot\infty=x$ then
$g^{-1}g'=T^B$ for some $B\in\QQ$, since $\Fix_H(\infty)$ is
contained in the image of $\QQ$ in $\Fix_G(\infty)$, by Proposition
\ref{prop:arith:cusps|fixinftycommensG1inQQ}. We have
$H_0^{g'}=(H_0^{g})^{T^B}$, so that $H_0^{g'}\cap\Fix_H(\infty)$ and
$H_0^g\cap\Fix_H(\infty)$ coincide, since $T^B$ centralizes
$\Fix_H(\infty)$. This shows that the value of $w_H(H_0\cdot x)$
doesn't depend upon the choice of $g\in H$ mapping $\infty$ to $x$.
Suppose that $H_0\cdot x=H_0\cdot x'$. Then $x'=h\cdot x$ for some
$h\in H_0$, and $g\cdot \infty=x$ implies $hg\cdot \infty=x'$. We
have $H_0^{hg}=H_0^g$, so the value of $w_H$ at the cusp $H_0\cdot
x$ is independent of the choice of representative point $x\in
\PP^1(\QQ)$.
\end{proof}
For $H$ and $H_0$ as above, and $C\in\mc{C}(H_0)$, we call $w_H(C)$
the {\em width of $C$ relative to $H$}. If $H=G_1$ then $w_H$
recovers the usual notion of width for cusps of finite index
subgroups of the modular group. Observe that if $H'<\PSL_2(\RR)$ is
commensurable with $G_1$ and contains $H$, then $H_0$ is a subgroup
of finite index in $H'$, and the functions $w_{H'}$ and $w_{H}$
coincide on $\mc{C}(H_0)$. Consequently, for a group
$H_0<\PSL_2(\RR)$ commensurable with $G_1$, to know all the widths
of a cusp $H_0\cdot x\in H_0\backslash\PP^1(\QQ)$, it suffices to
know the values $w_H(H_0\cdot x)$ for each supergroup $H>H'$ that is
maximal subject to being commensurable with $G_1$. The subgroups of
$\PSL_2(\RR)$ that are maximal subject to being commensurable with
$G_1$ are determined by Proposition \ref{prop:arith:trees|stabcell}.

Even though the functions $w_H$ and $w_{H'}$ in general don't
coincide for $H_0<H\cap H'$, the particular value $w_H(H_0\cdot
\infty)$ is easily checked to be independent of the choice of
supergroup $H$. Thus we can speak unambiguously of the {\em width of
$H_0$ at $\infty$}, to be denoted $w(H_0\cdot\infty)$, whenever
$H_0$ is commensurable with $G_1$.
\begin{prop}\label{prop:arith:cusps|widthbyorbs}
Suppose $H$ and $H_0$ are as in Proposition
\ref{prop:arith:cusps|cusps}, and suppose that $H$ is commensurable
with $G_1$. Then
\begin{gather}
     \frac{w_H(H_0h\cdot\infty)}{A}
     =\#(L_0\cdot h\Fix_H(\infty)).
\end{gather}
for all $h\in H$, where $A\in \QQp$ is the width of $H$ at $\infty$.
\end{prop}
\begin{proof}
Observe that the size of the orbit $L_0\cdot h\Fix_H(\infty)$ is
just the smallest $n\in\ZZp$ such that $T^{nA}\in\Fix_H(L_0\cdot
h)$. Observe also that $T^B$, for $b\in \QQ$, belongs to
$\Fix_H(L_0\cdot h)$ if and only if $T^B\in H_0^h$. For
$T^B\in\Fix_H(L_0\cdot h)$ if and only if $hT^B=h_0h$ for some
$h_0\in H_0$, and this occurs if and only if $T^B\in h^{-1}H_0h$.
Thus if $B=w_H(H_0h\cdot \infty)$, so that $B$ is the smallest
positive rational such that $T^B\in H_0^h$, then $B=nA$ where $n$ is
$\#(L_0\cdot h\Fix_H(\infty))$. This completes the proof.
\end{proof}

Consider the case that $H=G_1$ is the modular group, and
$\mc{O}=HC_N(1)$ is the hypercircle of hyperradius $N$ about $L_1$
for some positive integer $N$ (c.f. \S\ref{sec:arith:stabs}), and
take $L_0=L_N$, so that $H_0=G_{(1,N)}$. By Proposition
\ref{prop:arith:cusps|cusps}, the cusps of $G_{(1,N)}$ are in
natural correspondence with the orbits of $\Fix_G(\infty)\cap
G_1=\lab T\rab$ on $HC_N(1)$, and by Proposition
\ref{prop:arith:cusps|widthbyorbs}, the width of each cusp is just
the cardinality of the corresponding orbit of $\lab T\rab$.

It turns out that the orbit structure of $\lab T\rab$ on $HC_N(1)$
is not difficult to describe. We will not give the full analysis
here (since none of it would be new), but we will furnish the
following first step, which will be of use in
\S\ref{sec:arith:orbits}.

\begin{lem}\label{lem:arith:cusps|HCppow}
Suppose $N=p^n$ for some positive integer $n$. Then $HC_N(1)$
consists of the lattices of the form $p^{n-2a},k/p^a$ where $a$ and
$k$ satisfy one of the following conditions.
\begin{enumerate}
     \item     $a=0$ and $k=0$;
     \item     $0<a<n$ and $0<k<p^a$ and $\gcd\{k,p\}=1$;
     \item     $a=n$ and $0\leq k< p^n$.
\end{enumerate}
In other words, we have
\begin{gather}
     \begin{split}
     HC_{p^n}(1)&=
     \{p^n\}
     \cup
     \left\{
          \frac{p^{n-a}}{p^a},\frac{k}{p^a}\mid
          0<a<n,\;0<k<p^a,\;\gcd\{k,p\}=1\right\}\\
          &\qquad\cup
     \left\{\frac{1}{p^n},\frac{k}{p^n}\mid
          0\leq k<p^n\right\}.
     \end{split}
\end{gather}
\end{lem}

\subsection{Orbits}\label{sec:arith:orbits}

In this section we use Lemma \ref{lem:arith:cusps|HCppow} to furnish
a proof of Proposition \ref{prop:arith:stabs|GLtrans}. Since
$G\cong\PSL_2^+(\QQ)$ acts transitively on $P\mc{L}_1$, we have
verified Proposition \ref{prop:arith:stabs|GLtrans} as soon as we
show that $G_1\cong\PSL_2(\ZZ)$ acts transitively on the
hypercircles $HC_N(1)$ for all $N$. Actually, a stronger result is
true.
\begin{prop}
If $M$ and $N$ are positive integers with $\gcd\{M,N\}=1$, then the
group $G_{(1,M)}$ acts transitively on $HC_{N}(1)$.
\end{prop}
\begin{proof}
Consider first the case that $N=p^n$ is a prime power. Then $M$ is a
positive integer such that $p\nmid M$. The group $G_{(1,M)}$
contains both $T$ and $(T^M)^t$.
\begin{gather}
     T=\left[
         \begin{array}{cc}
           1 & 1 \\
           0 & 1 \\
         \end{array}
       \right],\quad
       (T^M)^t=\left[
         \begin{array}{cc}
           1 & 0 \\
           M & 1 \\
         \end{array}
       \right].
\end{gather}
Recall from Lemma \ref{lem:arith:cusps|HCppow} that $HC_{p^n}(1)$
consists of the lattices $L_{A,b}$ with $A=p^{n-a}/p^a$ and
$b=k/p^a$ for some $0\leq a\leq n$ and some $0\leq k<p^a$, with
$\gcd\{k,p^a\}=1$ in case $a<n$. Recall also, from
\S\ref{sec:arith:names}, that every projective lattice $L_{A,f/g}$
(for coprime positive integers $f,g$ with $0\leq f<g$) has a reverse
name $\bar{L}_{f'/g,1/g^2A}$, where $f'$ is the unique positive
integer less than $g$ such that $ff'\equiv 1\pmod{g}$. It is easy to
check then that $L_{A,b}$ belongs to $HC_{p^n}(1)$ just when
$\bar{L}_{b,A}$ does. The precise correspondence is as follows.
\begin{gather}
     \begin{split}\label{eqn:arith:orbits|revcorr}
     L_{p^n,0}&=\bar{L}_{0,1/p^n};\\
     L_{{p^{n-a}}/{p^a},{k}/{p^a}}
     &=\bar{L}_{k'/p^a,1/p^n},\quad p\nmid k,
          \;0<a<n;\\
     L_{1/p^n,k/p^a}&=\bar{L}_{k'/p^a,p^{n-a}/p^a},\quad p\nmid
     k,\;0<a\leq n;\\
     L_{1/p^n,0}&=\bar{L}_{0,p^n}.
     \end{split}
\end{gather}
In the above, $k'$ denotes the unique positive integer less than
$p^a$ such that $kk'\equiv 1\pmod{p^a}$. The reverse labels
$\bar{L}_{b,A}$ are convenient for describing the action of $T^t$.
For example, we have $\bar{L}_{b,A}\cdot (T^M)^t=\bar{L}_{b+AM,A}$.
In particular,
\begin{gather}
     \bar{L}_{0,1/p^n}\cdot (T^M)^t=\bar{L}_{M/p^n,1/p^n}.
\end{gather}
Since $M$ is coprime to $p$, we see that there is some power of
$T^t$ in $G_{(1,M)}$ that induces the permutation
\begin{gather}\label{eqn:arith:orbits|Ttperm}
     \bar{L}_{k/p^n,1/p^n}\mapsto \bar{L}_{(k+1)/p^n,1/p^n}
\end{gather}
on the lattices $\bar{L}_{b,A}$ in $HC_{p^n}(1)$ with $A=1/p^n$ and
$b=k/p^n$ for some $k$ satisfying $0\leq k<p^n$. Recall that $T\in
G_{(1,M)}$ induces a cyclic permutation on the lattices $L_{A,b}$ in
$HC_{p^n}(1)$ for the very same $A$ and $b$ --- viz. the permutation
obtained by removing the bars in (\ref{eqn:arith:orbits|Ttperm}).
Comparing with (\ref{eqn:arith:orbits|revcorr}) we see that the
$G_{(1,M)}$ orbit containing $L_{p^n,0}$ contains every lattice in
$HC_{p^n}(1)$; that is, $G_{(1,M)}$ acts transitively on
$HC_{p^n}(1)$ if $p\nmid M$.

More generally, consider the action of $G_{(1,M)}$ on $HC_{N}(1)$
for $N$ coprime to $M$. Then the elements of $HC_N(1)$ are in
natural correspondence with the elements of the cartesian product
\begin{gather}\label{eqn:arith:orbits|prodppows}
     HC_{p_1^{a_1}}(1)\times\cdots\times HC_{p_k^{a_k}}(1)
\end{gather}
for $N=p_1^{a_1}\cdots p_k^{a_k}$ a prime decomposition of $N$ (c.f.
\S\ref{sec:arith:trees}). The factors in
(\ref{eqn:arith:orbits|prodppows}) are sets of mutually coprime
order acted on transitively by $G_{(1,M)}$. It follows that
$G_{(1,M)}$ acts transitively on their product. This completes the
proof.
\end{proof}

\section{Diagrams}\label{sec:dyn} %

In this section we give our prescription for recovering McKay's
Monstrous $E_8$ observation (\ref{diag:intro_mdlrcorresp}) --- at
least, its reformulation in terms of discrete subgroups of
$\PSL_2(\RR)$ (c.f. \S\ref{sec:intro}) --- using elementary
properties of the group $\PSL_2(\RR)$.

\subsection{Setting}\label{sec:dyn:setting}

We adopt the setting of \S\ref{sec:arith} (more particularly, of
\S\ref{sec:arith:cusps}), with $\kk=\RR$, so that $V$ is an oriented
real vector space of dimension $2$, and ${\bf v}_1=(u_1,v_1)$ is an
ordered basis for $V$, and $V_1\subset V$ is the rational vector
space generated by $\{u_1,v_1\}$. The choice ${\bf v}_1\in \mc{B}^+$
(c.f. \S\ref{sec:arith:latts}) entails well-defined actions of the
groups $\PSL_2(\RR)$, $\PGL_2^+(\QQ)$, and $\PSL_2(\ZZ)$ on $V$, and
on various objects related to $V$ (c.f. \S\ref{sec:arith:stabs},
\S\ref{sec:arith:cusps}).

\subsection{Vertices}\label{sec:dyn:nodes}

If $\Gamma$ is a group and $\Gamma'$ is a finite index subgroup of
$\Gamma$, we write $[\Gamma:\Gamma']$ for the index of $\Gamma'$ in
$\Gamma$.

There is a well known formula for the index of $G_{(1,N)}$ (a.k.a
$\Gamma_0(N)$) in $G_1\cong \PSL_2(\ZZ)$. As demonstrated in
\cite{Con_UndrstndgGamma0N} it is easy to recover this formula by
considering the projections $\pi_p(L_N)$. For if $N=p_1^{a_1}\cdots
p_k^{a_k}$ is a prime decomposition of $N$, then, by Proposition
\ref{prop:arith:stabs|GLtrans}, the index of $G_{(1,N)}$ in $G_1$ is
the product over $i$ of the number of lattices at hyperdistance
$p_i^{a_i}$ from the distinguished lattice $L_1$ (c.f.
\S\ref{sec:arith:names}), since $\pi_p(L_N)=L_{p_i^{a_i}}$. The
cardinality of $HC_{p_i^{a_i}}(L_1)$ is $(p_i+1)p_i^{a_{i}-1}$ (by
Lemma \ref{lem:arith:cusps|HCppow}, for example). We thus obtain the
following expression,
\begin{gather}\label{eqn:dyn:nodes|indexformula}
     \left[G_1:G_{(1,N)}\right]
          =\prod_i(p_i+1)p_i^{a_i-1},
\end{gather}
which easily implies the following lemma.
\begin{lem}\label{lem:dyn:nodes|goodG1N}
If the index of $G_{(1,n)}$ in $G_1$ does not exceed $12$, then $n$
belongs to the following set.
\begin{gather}\label{list:dyn:nodes|ind12Ns}
     \{1,2,3,4,5,6,7,8,9,11\}
\end{gather}
\end{lem}
Given a group $\Gamma<\PSL_2(\RR)$ that is commensurable with
$G_1\cong\PSL_2(\ZZ)$, let us write $I^{G_1}_{\Gamma}$ for the index
of $\Gamma\cap G_1$ in $G_1$, and $I^{\Gamma}_{G_1}$ for the index
of $\Gamma\cap G_1$ in $\Gamma$.
\begin{equation}
     \xy
     (-23,0)*+{G_1}="G1";
     (23,0)*+{\Gamma}="Gamma";
     (0,-17)*+{\Gamma\cap G_1}="int";
     {\ar_*{I^{G_1}_{\Gamma}}@{-} "G1";"int"};
     {\ar^*{I_{G_1}^{\Gamma}}@{-} "Gamma";"int"};
     \endxy
\end{equation}
\begin{prop}\label{prop:dyn:nodes|nodes}
Suppose $\Gamma<\PSL_2(\RR)$ satisfies the following conditions.
\begin{enumerate}
\item     $\Gamma$ is arithmetic;\label{cond:dyn:nodes|arith}
\item     $\Gamma$ has width $1$ at $\infty$ (c.f.
\S\ref{sec:arith:cusps});\label{cond:dyn:nodes|width}
\item     there is some $N$ such that $\Gamma$ contains and normalizes
$G_{(1,N)}$, and the quotient $\Gamma/G_{(1,N)}$ is a group of
exponent $2$;\label{cond:dyn:nodes|exp2}
\item     $I_{\Gamma}^{G_1}\leq 12$ and
$I^{G_1}_{\Gamma}/I^{\Gamma}_{G_1}\leq
3$.\label{cond:dyn:nodes|ineq}
\end{enumerate}
Then $\Gamma$ is one of the groups in $\gt{E}$, where
\begin{gather}\label{list:dyn:nodes|nodegps}
     \gt{E}=
     \left\{
     G_{1},
     G_{(1,2)},
     G_{\{1,2\}},
     G_{\{1,3\}},
     G_{\{1,4\}},
     G_{\{1,5\}},
     G_{\{1,2,3,6\}},
     G_{\{2,4\}}^{(2)},
     G_{3}^{(3)}
     \right\}.
\end{gather}
\end{prop}
\begin{rmk}
Since $G_{(1,N)}$ and its normalizer are both commensurable with
$G_1$ for any $N$, condition \ref{cond:dyn:nodes|exp2} implies
condition \ref{cond:dyn:nodes|arith}.
\end{rmk}
\begin{rmk}
We may also write $G_{(1,N)+}$ for $G_{\{1,N\}}$ when
$N\in\{2,3,4,5\}$. We may write $G_{(1,6)+}$ for $G_{\{1,2,3,6\}}$.
The notation in (\ref{list:dyn:nodes|nodegps}) has been chosen
because it suggests, more strongly, how we can determine the correct
valence for each $\Gamma$ in $\gt{E}$ as a vertex in the affine
$E_8$ Dynkin diagram (cf. \S\ref{sec:dyn:edges:valency}).
\end{rmk}
\begin{proof}
By the Atkin--Lehner Theorem (Theorem \ref{thm:arith:stabs|AL}), the
intersection of the normalizer of $G_{(1,N)}$ with $G_1$ is
$G_{(1,N/h)}$ where $h$ is the largest divisor of $24$ such that
$h^2|N$. By condition \ref{cond:dyn:nodes|exp2} then, we have
$G_{(1,N)}<\Gamma\cap G_1<G_{(1,N/h)}$, for some $N$, with $h$ as in
the previous sentence. Equivalently, we have
\begin{gather}\label{eqn:dyn:nodes|boundingGamma1}
     G_{(1,hn)}<\Gamma\cap G_1<G_{(1,n)}
\end{gather}
for some $n$ and some $h$, where $h$ is a divisor of $\gcd\{n,24\}$
such that neither $4h$ nor $9h$ divide $n$.

By Lemma \ref{lem:dyn:nodes|goodG1N}, the only possibilities for the
$n$ in (\ref{eqn:dyn:nodes|boundingGamma1}) are those in the set
(\ref{list:dyn:nodes|ind12Ns}) if the first inequality of condition
\ref{cond:dyn:nodes|ineq} is to be satisfied; we will consider these
10 cases separately.

Let us agree to write $\Gamma_1$ for the intersection $\Gamma\cap
G_1$.
\begin{itemize}
\item{\it Case: $n=1$.} The inequality (\ref{eqn:dyn:nodes|boundingGamma1}) reduces
to $G_1<\Gamma_1<G_1$ in this case, so $\Gamma_1=G_1$. The
normalizer of $G_1$ is $G_1$, so $\Gamma=\Gamma_1=G_1$ in this case.
Conditions \ref{cond:dyn:nodes|arith} through
\ref{cond:dyn:nodes|ineq} are satisfied when $\Gamma=G_1$.

\item{\it Case: $n=2$.} If $n=2$ then $h\in\{1,2\}$.

If $h=1$ then $\Gamma_1=G_{(1,2)}$. The normalizer of $G_{(1,2)}$ is
$G_{\{1,2\}}$. The former group has index $2$ in the latter, so
$\Gamma$ is one of $G_{(1,2)}$ or $G_{\{1,2\}}$. The conditions
\ref{cond:dyn:nodes|arith} through \ref{cond:dyn:nodes|ineq} are
satisfied in both cases.

If $h=2$ then $\Gamma$ is assumed to normalize and contain
$G_{(1,4)}$, so that $G_{(1,4)}<\Gamma_1<G_{(1,2)}$, and
$\Gamma<G_2$.
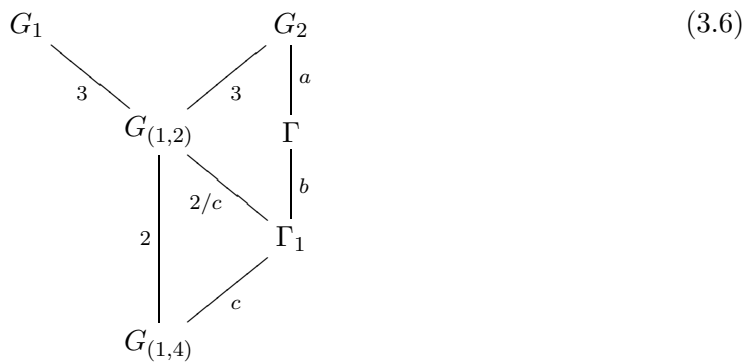
\begin{gather}\label{diag:dyn:nodes|N=4}
     \xymatrix{
     G_1 \ar@{-}[dr]_{3} && G_2 \ar@{-}[dl]^{3}\ar@{-}[d]^{a}\\
     & G_{(1,2)}\ar@{-}[dd]_{2}\ar@{-}[dr]_{2/c} & \Gamma \ar@{-}[d]^{b}\\
     &                             &\Gamma_1\ar@{-}[dl]^c\\
     & G_{(1,4)} &
     }
\end{gather}
The indices of these containments are as in
(\ref{diag:dyn:nodes|N=4}), for some $a,b,c$. Evidently, $abc=6$. In
order for $\Gamma/G_{(1,4)}$ to have exponent $2$ (c.f. condition
\ref{cond:dyn:nodes|exp2}), it must be that $a=3$ or $a=6$.

If $a=6$ then $\Gamma=\Gamma_1=G_{(1,4)}$, and
$I^{G_1}_{\Gamma}/I^{\Gamma}_{G_1}=I^{G_1}_{\Gamma}=6$, in violation
of condition \ref{cond:dyn:nodes|ineq}. So $a=3$, and $bc=2$. If
$c=2$ then $b=1$ and $\Gamma=\Gamma_1=G_{(1,2)}$, and we have seen
already that this group satisfies conditions
\ref{cond:dyn:nodes|arith} through \ref{cond:dyn:nodes|ineq}. If
$b=2$ and $c=1$, then $\Gamma$ is a subgroup of index $3$ in $G_2$
whose intersection with $G_1$ is exactly $G_{(1,4)}$. The group
$G_2$ acts by permutations on $HC_2(2)$ --- the hypercircle of
hyperradius $2$ about $L_2$ --- and $G_{(1,4)}$ is exactly the
subgroup stabilizing $L_4$. It follows that the natural map
$G_2\to\Sym(HC_2(2))$ is surjective, and the image of $\Gamma$ under
this map is the group generated by one transposition.
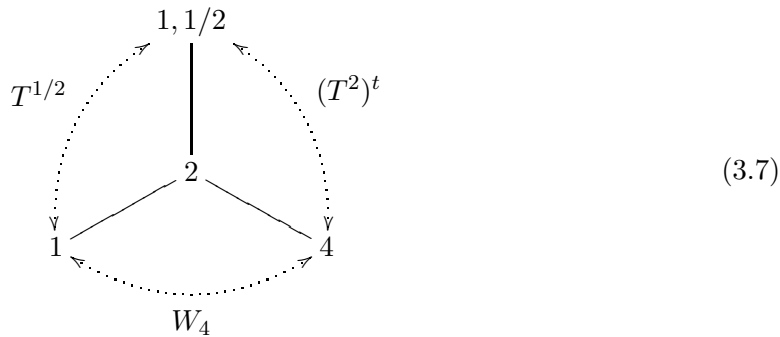
\begin{equation}\label{diag:dyn:nodes|HC2(2)}
     \xy
     (0,20)*+{{1,{1}/{2}}}="1h";
     (0,0)*+{2}="2";
     (18,-10)*+{4}="4";
     (-18,-10)*+{1}="1";
     {\ar@{-} "1h";"2"};
     {\ar@{-} "1";"2"};
     {\ar@{-} "4";"2"};
     {\ar^*+{(T^2)^t} @/^1.5pc/ @{<.>} "1h";"4"};
     {\ar^*+{T^{1/2}} @/^1.5pc/ @{<.>} "1";"1h"};
     {\ar_*+{W_4} @/_1.5pc/ @{<.>} "1";"4"};
     \endxy
\end{equation}
Diagram (\ref{diag:dyn:nodes|HC2(2)}) displays the smallest subtree
of the $2$-adic tree (c.f. \S\ref{sec:arith:trees}) containing
$HC_2(2)$, together with elements of $G_2$ that give rise to each of
the transpositions in $\Sym (HC_2(2))$. (We may take
$W_4=g^{1/2}Sg^{2}$.) Our group $\Gamma$ does not contain $T^{1/2}$
by condition \ref{cond:dyn:nodes|width}, and it does not contain
$(T^2)^t$ since this element lies in $G_{1}$ but not in
$\Gamma_1=G_{(1,4)}$. The remaining possibility is that $\Gamma=\lab
G_{(1,4)},W_4\rab=G_{\{1,4\}}$. This group satisfies conditions
\ref{cond:dyn:nodes|arith} through \ref{cond:dyn:nodes|ineq}.

\item{\it Case: $n=3$.} If $n=3$ then $h\in\{1,3\}$.

If $h=1$ then $\Gamma_1=G_{(1,3)}$ and
     $\Gamma<G_{\{1,3\}}$. The group $G_{(1,3)}$ has index $2$ in
     $G_{\{1,3\}}$, so either $\Gamma=G_{(1,3)}$ or
     $\Gamma=G_{\{1,3\}}$. In the former case the second inequality of
     condition \ref{cond:dyn:nodes|ineq} is violated. In the latter
     case conditions \ref{cond:dyn:nodes|arith} through
     \ref{cond:dyn:nodes|ineq} are all satisfied.

If $h=3$ then $\Gamma$ is assumed to normalize and contain
$G_{(1,9)}$, so that $G_{(1,9)}<\Gamma_1<G_{(1,3)}$, and
$\Gamma<G_{3}$.
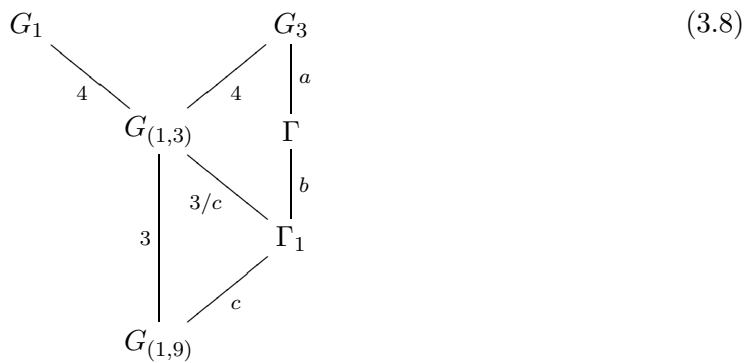
\begin{gather}\label{diag:dyn:nodes|N=9}
     \xymatrix{
     G_1 \ar@{-}[dr]_{4} && G_3 \ar@{-}[dl]^{4}\ar@{-}[d]^{a}\\
     & G_{(1,3)}\ar@{-}[dd]_{3}\ar@{-}[dr]_{3/c} & \Gamma \ar@{-}[d]^{b}\\
     &                             &\Gamma_1\ar@{-}[dl]^c\\
     & G_{(1,9)} &
     }
\end{gather}
The indices of these containments are as in
(\ref{diag:dyn:nodes|N=9}), for some $a,b,c$. Evidently, $abc=12$.
In order for $\Gamma/G_{(1,9)}$ to have exponent $2$ (c.f. condition
\ref{cond:dyn:nodes|exp2}), it must be that $3|a$, and in
particular, we must have $c=1$; i.e. $\Gamma_1=G_{(1,9)}$. In order
that the second inequality of condition \ref{cond:dyn:nodes|ineq} be
satisfied, we require that $b\geq 4$. Consequently, $\Gamma$ is a
subgroup of index $3$ in $G_3$ whose intersection with $G_1$ is
exactly the group $G_{(1,9)}$.

The group $G_3$ acts by permutations on $HC_3(3)$, and $G_{(1,9)}$
is exactly the kernel of the corresponding map $G_3\to\Sym(HC_3(3))$
(c.f. \S\ref{sec:arith:chars:N=9}). The image of $G_3$ under this
map is a copy of the alternating group on $4$ symbols. The image of
$\Gamma$ in $\Sym(HC_3(3))$ must be a subgroup of order $4$ in this
$\Alt_4$; such a subgroup is unique --- it is the kernel of any
non-trivial homomorphism $\Alt_4\to\ZZ/3$. We conclude that
$\Gamma=G_{3}^{(3)}$ (c.f. \S\ref{sec:arith:chars:N=9}). This group
satisfies conditions \ref{cond:dyn:nodes|arith} through
\ref{cond:dyn:nodes|ineq}.

\item{\it Case: $n=4$.} If $n=4$ then $h\in\{2,4\}$.

If $h=2$ then $G_{(1,8)}<\Gamma_1<G_{(1,4)}$ and
$\Gamma<G_{\{2,4\}}$.
\begin{gather}\label{diag:dyn:nodes|N=8}
     \xymatrix{
         G_1\ar@{-}[ddr]_6 &&G_{\{2,4\}}\ar@{-}[dl]^2 \ar@{-}[d]^a\\
     & G_{(2,4)} \ar@{-}[d]^{2}&\Gamma\ar@{-}[dd]^{b}\\
    & G_{(1,4)}\ar@{-}[dd]_{2}\ar@{-}[dr]_{2/c} &\\
     &                             &\Gamma_1\ar@{-}[dl]^c\\
     & G_{(1,8)} &
     }
\end{gather}
We have $abc=8$ for $a,b,c$ as in (\ref{diag:dyn:nodes|N=8}). The
group $G_{\{2,4\}}$ does not have width $1$ at $\infty$, so $a\geq
2$. On the other hand, by condition \ref{cond:dyn:nodes|ineq} we
have $I^{G_1}_{\Gamma}/I^{\Gamma}_{G_1}=12/bc\leq 3$, so $bc\geq 4$,
so we conclude $a=2$ and $bc=4$.

The group $G_{\{2,4\}}$ acts by permutations on the set
$S=HC_2(2)\cup HC_2(4)$, and the image of $G_{\{2,4\}}$ in $\Sym(S)$
is a group of the shape
\begin{gather}\label{eqn:dyn:nodes|G24plusmodG18}
     (\ZZ/2\times\ZZ/2)\rtimes\ZZ/2
\end{gather}
(i.e. a Dihedral group of order $8$) where the central factors
$\ZZ/2$ are generated by $T^{1/2}$ and $(T^4)^t$, and a non-central
$\ZZ/2$ is generated by the Atkin--Lehner involution $W_8$ (c.f.
\S\ref{sec:arith:chars:N=8}). The group $G_{(1,8)}$ is the kernel of
the natural map $G_{\{2,4\}}\to\Sym(S)$. Consequently, the image of
$\Gamma$ in $\Sym(S)$ is a subgroup of order four in
(\ref{eqn:dyn:nodes|G24plusmodG18}). By condition
\ref{cond:dyn:nodes|width} it does not contain $T^{1/2}$. By
condition \ref{cond:dyn:nodes|exp2} it is not cyclic. There is
exactly one possibility: $\Gamma=\lab G_{(1,8)}, (T^4)^tT^{1/2},
W_8\rab$; i.e. $\Gamma=G_{\{2,4\}}^{(2)}$ (c.f.
\S\ref{sec:arith:chars:N=8}). This group satisfies all the required
properties.

If $h=4$ then $G_{(1,16)}<\Gamma_1<G_{(1,4)}$ and $\Gamma<G_{4}$.
\begin{gather}\label{diag:dyn:nodes|N=16}
     \xymatrix{
     G_1 \ar@{-}[dr]_{6} && G_4 \ar@{-}[dl]^{6}\ar@{-}[d]^{a}\\
     & G_{(1,4)}\ar@{-}[dd]_{4}\ar@{-}[dr]_{4/c} & \Gamma \ar@{-}[d]^{b}\\
     &                             &\Gamma_1\ar@{-}[dl]^c\\
     & G_{(1,16)} &
     }
\end{gather}
We have $abc=24$ for $a,b,c$ as in (\ref{diag:dyn:nodes|N=16}). By
condition \ref{cond:dyn:nodes|exp2}, we have $3|a$. By condition
\ref{cond:dyn:nodes|ineq} we have
$I^{G_1}_{\Gamma}/I^{\Gamma}_{G_1}=24/bc\leq 3$, so $bc\geq 8$, and
this implies $a=3$ and $bc=8$.
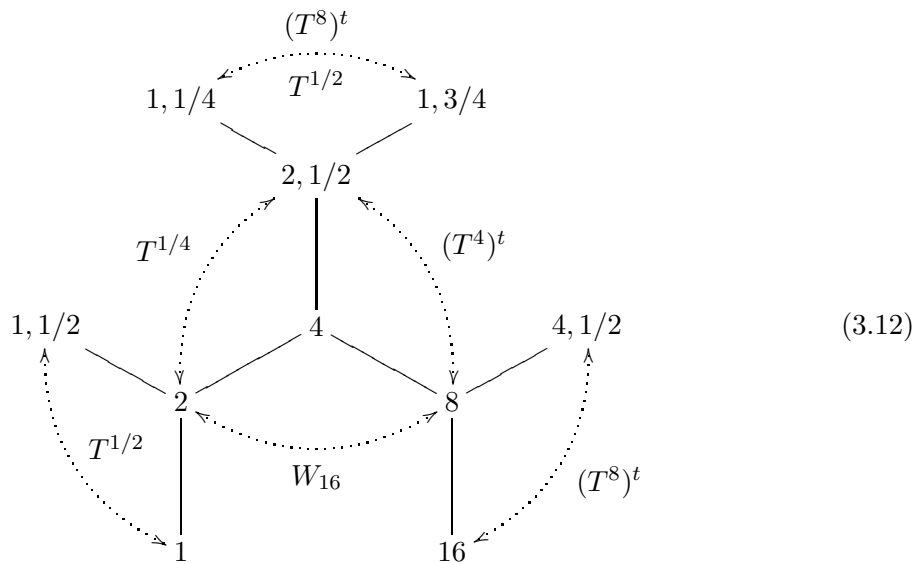
\begin{equation}\label{diag:dyn:nodes|HC4(4)}
     \xy
     (-18,30)*+{1,1/4}="1q";
     (18,30)*+{1,3/4}="1qqq";
     (0,20)*+{2,1/2}="2h";
     (0,0)*+{4}="4";
     (18,-10)*+{8}="8";
     (18,-30)*+{16}="16";
     (-18,-10)*+{2}="2";
     (-18,-30)*+{1}="1";
     (-36,0)*+{1,1/2}="1h";
     (36,0)*+{4,1/2}="4h";
     {\ar@{-} "2h"; "4"};
     {\ar@{-} "2"; "4"};
     {\ar@{-} "8"; "4"};
     {\ar@{-} "1h"; "2"};
     {\ar@{-} "1"; "2"};
     {\ar@{-} "4h"; "8"};
     {\ar@{-} "16"; "8"};
     {\ar@{-} "1q"; "2h"};
     {\ar@{-} "1qqq"; "2h"};
     {\ar^*+{(T^4)^t} @/^1.5pc/ @{<.>} "2h";"8"};
     {\ar^*+{T^{1/4}} @/^1.5pc/ @{<.>} "2";"2h"};
     {\ar_*+{W_{16}} @/_1.5pc/ @{<.>} "2";"8"};
     {\ar_*+{T^{1/2}} @/^1.5pc/ @{<.>} "1"; "1h"};
     {\ar_*+{T^{1/2}} @/^1.5pc/ @{<.>} "1q"; "1qqq"};
     {\ar^*+{(T^{8})^t} @/^1.5pc/ @{<.>} "4h"; "16"};
     (0,40)*+{(T^8)^t}="T8t";
     \endxy
\end{equation}
The diagram in (\ref{diag:dyn:nodes|HC4(4)}) shows the part of the
$2$-adic tree containing $HC_4(4)$; the group $G_4$ acts as
automorphisms of this tree, and the subgroup $G_{(1,16)}$ fixes
every node. We see from (\ref{diag:dyn:nodes|HC4(4)}) that the
quotient $G_4/G_{(1,16)}$ has the structure
$(\ZZ/2\times\ZZ/2)\rtimes\Sym_3$, with the elementary abelian
subgroup generated by $T^{1/2}$ and $(T^{8})^t$, and the symmetric
group acting on it generated by $T^{1/4}$ and $(T^{4})^t$ (c.f.
(\ref{diag:dyn:nodes|HC2(2)})). If $\Gamma$ is a subgroup of $G_4$
of order $8$ containing $G_{(1,16)}$, then the image of $\Gamma$ in
$G_4/G_{(1,16)}$ contains the subgroup $\ZZ/2\times\ZZ/2$. It
follows that $\Gamma$ contains $T^{1/2}$, but this violates
condition \ref{cond:dyn:nodes|width}.

\item{\it Case: $n=5$.} If $n=5$ then $h=1$ and $\Gamma_1=G_{(1,5)}$,
and $G_{(1,5)}<\Gamma<G_{\{1,5\}}$. Similar to the case that
$N=nh=3$, the only possibility satisfying conditions
\ref{cond:dyn:nodes|arith} through \ref{cond:dyn:nodes|ineq} is
$\Gamma=G_{\{1,5\}}$.

\item{\it Case: $n=6$.} If $n=6$ then $h\in\{1,2,3\}$.

If $h=1$ then $\Gamma_1=G_{(1,6)}$, and $\Gamma<G_{\{1,2,3,6\}}$.
The group $G_{(1,6)}$ has index $12$ in $G_1$, and index $4$ in
$G_{\{1,2,3,6\}}$, so in order to satisfy condition
\ref{cond:dyn:nodes|ineq}, it must be that $\Gamma=G_{\{1,2,3,6\}}$.
This group satisfies all the required properties.

If $h=2$ then $G_{(1,12)}<\Gamma_1<G_{(1,6)}$, and
$\Gamma<G_{\{2,6\}}$. The group $G_{(1,6)}$ has index $12$ in $G_1$,
so it must be that $\Gamma_1=G_{(1,6)}$, by the first inequality of
condition \ref{cond:dyn:nodes|ineq}.
\begin{gather}\label{diag:dyn:nodes|N=12}
     \xymatrix{
         G_1\ar@{-}[ddr]_{12} &&G_{\{2,6\}}\ar@{-}[dl]^2 \ar@{-}[d]^a\\
     & G_{(2,6)} \ar@{-}[d]^{2}&\Gamma\ar@{-}[dl]^{b}\\
    & G_{(1,6)}=\Gamma_1\ar@{-} &\\
     }
\end{gather}
We have $ab=4$ for $a,b$ as in (\ref{diag:dyn:nodes|N=12}). For
condition \ref{cond:dyn:nodes|ineq} we require $b\geq 4$, so that
$\Gamma=G_{\{2,6\}}$, but this group does not satisfy condition
\ref{cond:dyn:nodes|width}. Diagram (\ref{diag:dyn:nodes|26cell})
shows the cell $\{L_2,L_6\}$ together with sufficiently many
adjacent nodes (in the $2$ and $3$-adic trees) so as to allow us to
display the action of $G_{\{2,6\}}/G_{(1,6)}$ (and verify the
indices in (\ref{diag:dyn:nodes|N=12})).
\begin{equation}\label{diag:dyn:nodes|26cell}
     \xy
     (-20,0)*+{3}="3";
     (0,0)*+{6}="6";
     (20,0)*+{12}="12";
     (-20,-25)*+{1}="1";
     (0,-25)*+{2}="2";
     (20,-25)*+{4}="4";
     {\ar@{-} "3"; "6"};
     {\ar@{-} "6"; "12"};
     {\ar@{-} "1"; "2"};
     {\ar@{-} "2"; "4"};
     {\ar@{=} "3"; "1"};
     {\ar@{=} "6"; "2"};
     {\ar@{=} "12"; "4"};
     {\ar_*+{W_2} @/_1.5pc/ @{<.>} "1";"4"};
     {\ar_*+{W_3} @/_1.5pc/ @{<.>} "3";"1"};
     \endxy
\end{equation}

If $h=3$ then $G_{(1,18)}<\Gamma_1<G_{(1,6)}$ and
$\Gamma<G_{\{3,6\}}$. Again, we must have $\Gamma_1=G_{(1,6)}$, by
condition \ref{cond:dyn:nodes|ineq}.
\begin{gather}\label{diag:dyn:nodes|N=18}
     \xymatrix{
         G_1\ar@{-}[ddr]_{12} &&G_{\{3,6\}}\ar@{-}[dl]^2 \ar@{-}[d]^a\\
     & G_{(3,6)} \ar@{-}[d]^{2}&\Gamma\ar@{-}[dl]^{b}\\
    & G_{(1,6)}=\Gamma_1\ar@{-} &\\
     }
\end{gather}
Considering indices (c.f. (\ref{diag:dyn:nodes|N=18})) we see that
$\Gamma$ must coincide with $G_{\{3,6\}}$ if condition
\ref{cond:dyn:nodes|ineq} is to be satisfied, but this group fails
to satisfy condition \ref{cond:dyn:nodes|width}. Diagram
(\ref{diag:dyn:nodes|36cell}) is the analogue of
(\ref{diag:dyn:nodes|26cell}) for $h=3$.
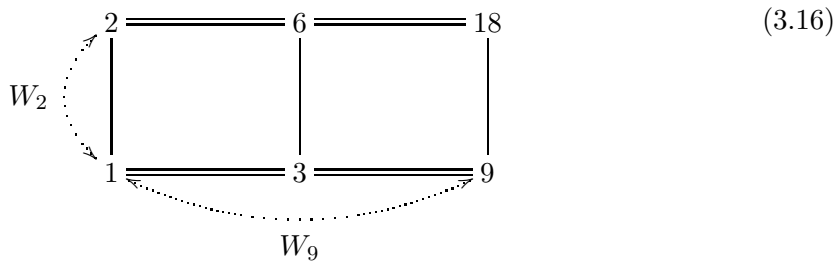
\begin{equation}\label{diag:dyn:nodes|36cell}
     \xy
     (-25,0)*+{2}="2";
     (0,0)*+{6}="6";
     (25,0)*+{18}="18";
     (-25,-20)*+{1}="1";
     (0,-20)*+{3}="3";
     (25,-20)*+{9}="9";
     {\ar@{-} "1"; "2"};
     {\ar@{-} "3"; "6"};
     {\ar@{-} "9"; "18"};
     {\ar@{=} "1"; "3"};
     {\ar@{=} "3"; "9"};
     {\ar@{=} "2"; "6"};
     {\ar@{=} "6"; "18"};
     {\ar^*+{W_2} @/^1.5pc/ @{<.>} "1";"2"};
     {\ar_*+{W_9} @/_1.5pc/ @{<.>} "1";"9"};
     \endxy
\end{equation}

\item{\it Case: $n=7$.} If $n=7$ then $h=1$, so we have $\Gamma_1=G_{(1,7)}$ and
$\Gamma<G_{\{1,7\}}$. Then $G_{(1,7)}$ has index $8$ in $G_1$, and
index $2$ in $G_{\{1,7\}}$, so there is no possibility for $\Gamma$
that satisfies condition \ref{cond:dyn:nodes|ineq}.

\item{\it Case: $n=8$.} If $n=8$ then $h\in\{4,8\}$. Since
$G_{(1,8)}$ has index $12$ in $G_1$, we have $\Gamma_1=G_{(1,8)}$.
The group $\Gamma$ is assumed to normalize $G_{(1,8h)}$, and the
quotient $\Gamma/G_{(1,8h)}$ should have exponent $2$, by condition
\ref{cond:dyn:nodes|exp2}. Observe that $G_{(1,8)}/G_{(1,8h)}$ is a
cyclic group of order $h$ --- we may take the image of $(T^8)^t$ as
a generator. Since $h>2$, there are no groups in $\gt{E}$ that arise
for $n=8$.

\item{\it Case: $n=9$.} For $n=9$ we have $h\in\{3,9\}$. Since
$G_{(1,9)}$ has index $12$ in $G_1$, we have $\Gamma_1=G_{(1,9)}$.
The group $\Gamma$ is assumed to normalize $G_{(1,9h)}$, and, by
condition \ref{cond:dyn:nodes|exp2}, it is required that
$\Gamma/G_{(1,9h)}$ have exponent $2$, but $G_{(1,9h)}$ has index
$h$ in $\Gamma_1$. We conclude that no elements of $\gt{E}$ arise
for $n=9$.

\item{\it Case: $n=11$.} Similar to the case that $n=7$, we must
have $\Gamma_1=G_{(1,11)}$ and $G_{(1,11)}<\Gamma<G_{\{1,11\}}$. But
$G_{(1,11)}$ has index $12$ in $G_1$, and index $2$ in
$G_{\{1,11\}}$, so there is no possibility for $\Gamma$ that
satisfies condition \ref{cond:dyn:nodes|ineq}.
\end{itemize}
We have accounted for all the groups appearing in the set $\gt{E}$,
and we have verified that no other subgroups of $\PSL_2(\RR)$
satisfy conditions \ref{cond:dyn:nodes|arith} through
\ref{cond:dyn:nodes|ineq}.
\end{proof}

\subsection{Edges}\label{sec:dyn:edges}

The groups in $\gt{E}$ are exactly those that label the vertices in
the diagram (\ref{diag:intro_mdlrcorresp}). We now seek to
reconstruct the edges of the diagram (\ref{diag:intro_mdlrcorresp}),
by examination of the properties of the groups in $\gt{E}$.

\subsubsection{Thread}\label{sec:dyn:edges:thread}

By condition \ref{cond:dyn:nodes|exp2} of Proposition
\ref{prop:dyn:nodes|nodes}, each group $\Gamma\in\gt{E}$ satisfies
\begin{gather}\label{eqn:dyn:edges:level|normlzr}
     G_{(1,N)}<\Gamma<N_G\left(G_{(1,N)}\right)=G_{(h,N/h)+}
\end{gather}
for some positive integer $N$ (where $h$ is the largest divisor of
$24$ such that $h^2$ divides $N$ --- c.f. Theorem
\ref{thm:arith:stabs|AL}). For $\Gamma\in\gt{E}$ let $N_{\Gamma}$ be
the minimal $N$ with this property. (The group $\Gamma=G_{(1,2)}$
satisfies (\ref{eqn:dyn:edges:level|normlzr}) for $N=2$ and $N=4$.)
Now let $a_{\Gamma}$ be the largest divisor of $24$ such that
$a_{\Gamma}^2$ divides $N_{\Gamma}$, and $G_{(a,N/a)}$ contains
$\Gamma\cap G_{(a,N/a)}$ to index $a$, for $a=a_{\Gamma}$ and
$N=N_{\Gamma}$.
\begin{equation}\label{diag:dyn:edges:level|as}
     \xy
     (0,0)*+{G_{(h,N/h)+}}="Ghpl";
     (-18,-10)*+{G_{(h,N/h)}}="Gh";
     (18,-10)*+{\Gamma}="Gamma";
     (-18,-30)*+{G_{(a,N/a)}}="Ga";
     (18,-30)*+{\Gamma\cap G_{(a,N/a)}}="Gammaa";
     (0,-40)*+{G_{(1,N)}}="G1N";
     {\ar@{-} "Ghpl"; "Gh"};
     {\ar@{-} "Ghpl"; "Gamma"};
     {\ar@{-} "Gh"; "Ga"};
     {\ar@{-}^*+{a} "Gamma"; "Gammaa"};
     {\ar@{-} "Ga"; "G1N"};
     {\ar@{-} "Gammaa"; "G1N"};
     \endxy
\end{equation}
\begin{lem}\label{lem:dyn:nodes|as}
If $\Gamma\neq G_{\{1,4\}}$ then $a_{\Gamma}$ is the largest divisor
of $24$ such that $a_{\Gamma}^2$ divides $N_{\Gamma}$; that is,
$a=a_{\Gamma}=h$ in (\ref{diag:dyn:edges:level|as}). If
$\Gamma=G_{\{1,4\}}$ then $a_{\Gamma}=1$.
\end{lem}
For each $\Gamma\in\gt{E}$ define a subgroup $\Gamma_0<\Gamma$ by
setting $\Gamma_0=\Gamma\cap G_{(a,N/a)}$, where $a=a_{\Gamma}$ and
$N=N_{\Gamma}$. If we agree to write also $G_{(1,N)}^{(1)}$ for
$G_{(1,N)}$, and $G_{(a,a)}^{(a)}$ for $G_{a}^{(a)}$, then we have
\begin{gather}
     \Gamma_0=G_{(a,N/a)}^{(a)}
\end{gather}
for $a=a_{\Gamma}$ and $N=N_{\Gamma}$, for each $\Gamma\in\gt{E}$.

\medskip

The groups $\Gamma_0$, for $\Gamma\in\gt{E}$, may be recovered in a
more geometric way as follows.

Relax for a moment the assumption that $\Gamma\in\gt{E}$, and
suppose, more generally, that $\Gamma$ is an arithmetic subgroup of
$\PSL_2(\RR)$ such that $\Gamma$ contains and normalizes some
$G_{(1,N)}$. Then $\Gamma$ stabilizes the set $(1,N)+$, and the
normalizer of $\Gamma$ stabilizes some subset of $(1,N)+$.

Reinstate now the assumption that $\Gamma\in\gt{E}$, and define
$S_{\Gamma}$ to be the largest subset of the $(1,N)$-thread that is
stabilized by $N_G(\Gamma)$. For each $\Gamma\in\gt{E}$ except for
$\Gamma=G_{\{1,4\}}$, the normalizer $N_G(\Gamma)$, of $\Gamma$, is
a maximal discrete subgroup of $\PSL_2(\RR)$, and therefore, by
Proposition \ref{prop:arith:trees|stabcell}, stabilizes a unique
cell in $P\mc{L}_1$. The set $S_{\Gamma}$ then recovers this cell.
In the case that $\Gamma=G_{\{1,4\}}$, we have $N_G(\Gamma)=\Gamma$,
while the group $G_2$ properly contains $\Gamma$. We have
$S_{{\Gamma}}=\{L_1,L_2,L_4\}$ (c.f.
(\ref{diag:arith:nodes|threads})) for $\Gamma=G_{\{1,4\}}$.

For $\Gamma\in\gt{E}$ consider the subgroup of $N_G({\Gamma})$
fixing the elements of $S_{{\Gamma}}$ point-wise; this group is
exactly $\Gamma_0$, for each $\Gamma\in\gt{E}$.

\medskip

The following lemma is easily checked
--- by inspection of the first two lines of Table
\ref{table:dyn:edges|vals}, for example.
\begin{lem}\label{lem:dyn:edges:thread|Gamma0normal}
For each $\Gamma\in\gt{E}$, the subgroup $\Gamma_0$ is contained
normally in $\Gamma$, and the quotient $\Gamma/\Gamma_0$ is a group
of exponent $2$.
\end{lem}

The group $\Gamma_0$ may be regarded as the group we obtain from
$\Gamma$ by removing its Atkin--Lehner involutions.

\subsubsection{Level}\label{sec:dyn:edges:level}

Let $\Gamma$ be an arithmetic subgroup of $\PSL_2(\RR)$, and suppose
that $\Gamma$ is a congruence group (c.f. \S\ref{sec:arith:stabs}).
Then $\Gamma$ has a well-defined level; viz. the minimal positive
integer $N$ such that $\Gamma(N)<\Gamma$. Let us write
$\lev(\Gamma)$ for the level of an arithmetic congruence group
$\Gamma$. For $\Gamma\in\gt{E}$ we define the {\em normalized level}
of $\Gamma$, to be denoted $\levo(\Gamma)$, by setting
\begin{gather}
     \levo(\Gamma)=\frac{\lev(\Gamma)}{a_{\Gamma}}
\end{gather}
where $a_{\Gamma}$ is as in \S\ref{sec:dyn:edges:thread}.
\begin{table}[ht]
   \centering
   \caption{$\Gamma$, $a_{\Gamma}$ and $\levo(\Gamma)$, for $\Gamma\in\gt{E}$}
   \label{table:dyn:edges|levos}
     \begin{tabular}{|c|c|c|c|c|c|c|c|c|}
       \hline
       $G_1$ & $G_{\{1,2\}}$ & $G_{\{1,3\}}$ & $G_{\{1,4\}}$ & $G_{\{1,5\}}$ & $G_{\{1,2,3,6\}}$ & $G^{(3)}_3$ & $G_{\{2,4\}}^{(2)}$ & $G_{(1,2)}$ \\
       $1$ & $1$ & $1$ & $1$ & $1$ & $1$ & $3$ & $2$ & $1$ \\
       $1$ & $2$ & $3$ & $4$ & $5$ & $6$ & $3$ & $4$ & $2$ \\
       \hline
     \end{tabular}
\end{table}
The values $a_{\Gamma}$ and the normalized levels $\levo(\Gamma)$
are displayed in Table \ref{table:dyn:edges|levos}.

\subsubsection{Valency}\label{sec:dyn:edges:valency}

By Lemma \ref{lem:dyn:edges:thread|Gamma0normal}, the group
$\Gamma_0$ is normal in $\Gamma$ for each $\Gamma\in\gt{E}$, and the
quotient $\Gamma/\Gamma_0$ has exponent $2$. For $\Gamma\in\gt{E}$
we define the {\em valency} of $\Gamma$, to be denoted
$\val(\Gamma)$, by setting $\val(\Gamma)=m+1$, where $m$ satisfies
\begin{gather}
     \Gamma/\Gamma_0\cong (\ZZ/2)^{m}.
\end{gather}
The value $\val(\Gamma)$ encodes the number of Atkin--Lehner
involutions in $\Gamma$. The valencies of the groups in $\gt{E}$ are
recorded in Table \ref{table:dyn:edges|vals}.
\begin{table}[h]
   \centering
   \caption{$\Gamma$, $\Gamma_0$ and $\val(\Gamma)$, for $\Gamma\in\gt{E}$}
   \label{table:dyn:edges|vals}
     \begin{tabular}{|c|c|c|c|c|c|c|c|c|}
       \hline
       $G_1$ & $G_{\{1,2\}}$ & $G_{\{1,3\}}$ & $G_{\{1,4\}}$ & $G_{\{1,5\}}$ & $G_{\{1,2,3,6\}}$ & $G^{(3)}_3$ & $G_{\{2,4\}}^{(2)}$ & $G_{(1,2)}$ \\
       $G_1$ & $G_{(1,2)}$ & $G_{(1,3)}$ & $G_{(1,4)}$ & $G_{(1,5)}$ & $G_{(1,6)}$ & $G^{(3)}_3$ & $G_{(2,4)}^{(2)}$ & $G_{(1,2)}$ \\
       $1$ & $2$ & $2$ & $2$ & $2$ & $3$ & $1$ & $2$ & $1$ \\
       \hline
     \end{tabular}
\end{table}

\subsubsection{Faithfulness}\label{sec:dyn:edges:faith}

Recall that in the classical McKay Correspondence for the binary
icosahedral group $2.\Alt_5$, the nodes in the affine $E_8$ Dynkin
diagram are labeled by irreducible representations of $2.\Alt_5$. In
particular, four of the nodes correspond to faithful representations
of $2.\Alt_5$, and the remaining five nodes correspond to
representations that factor through $\Alt_5$.

With this in mind, we will say that a group $\Gamma\in\gt{E}$ is
{\em faithful} if $\Gamma$ is ``not far'' from $G_{\{1,2\}}$, in the
sense that $[\Gamma,\Gamma\cap G_{\{1,2\}}]\leq 2$. We set
$\gt{E}_1$ to be the subset of groups in $\gt{E}$ that are faithful,
and we set $\gt{E}_0=\gt{E}\setminus\gt{E}_1$.
\begin{gather}
     \gt{E}_1=\left\{
          G_{\{1,2\}},G_{\{1,4\}},G_{\{1,2,3,6\}},G_{(1,2)}
               \right\}\\
     \gt{E}_0=\left\{
          G_1,G_{\{1,3\}},G_{\{1,5\}},G_{3}^{(3)},G_{\{2,4\}}^{(2)}
               \right\}
\end{gather}

\subsubsection{Prescription}\label{sec:dyn:edges:pres}

The following proposition is easily checked.
\begin{prop}\label{prop:dyn:edges|edges}
There is a unique graph with vertex set $\gt{E}$ satisfying the
following properties.
\begin{itemize}
\item     The valence of $\Gamma\in\gt{E}$ is $\val(\Gamma)$.
\item     The identity $2\levo(\Gamma)=\sum_{\Gamma'\in\adj(\Gamma)}\levo(\Gamma')$
holds for all $\Gamma\in\gt{E}$, where $\adj(\Gamma)$ denotes the
set of vertices that are adjacent to $\Gamma$.
\item     If $\Gamma\in\gt{E}_1$ then $\adj(\Gamma)\subset\gt{E}_0$.
\end{itemize}
\end{prop}
The graph whose existence and uniqueness is guaranteed by
Proposition \ref{prop:dyn:edges|edges} is displayed in
(\ref{diag:dyn:edges|graph}).
\begin{equation}\label{diag:dyn:edges|graph}
     \xy
     (0,0)*+{G_1}="1A";
     (0,0)*\xycircle(3,2){-};
     (15,0)*+{G_{\{1,2\}}}="2A";
     (30,0)*+{G_{\{1,3\}}}="3A";
     (45,0)*+{G_{\{1,4\}}}="4A";
     (60,0)*+{G_{\{1,5\}}}="5A";
     (75,0)*+{G_{\{1,2,3,6\}}}="6A";
     (85,10)*+{G_3^{(3)}}="3C";
     (87,-8)*+{G_{\{2,4\}}^{(2)}}="4B";
     (99,-16)*+{G_{(1,2)}}="2B";
     {\ar@{-} "1A";"2A"};
     {\ar@{-} "2A";"3A"};
     {\ar@{-} "3A";"4A"};
     {\ar@{-} "4A";"5A"};
     {\ar@{-} "5A";"6A"};
     {\ar@{-} "6A";"3C"};
     {\ar@{-} "6A";"4B"};
     {\ar@{-} "4B";"2B"};
     \endxy
\end{equation}
In terms of the more standard notation for the discrete groups of
Monstrous Moonshine, this is exactly the diagram
(\ref{diag:intro_mdlrcorresp}) we sought to recover.

\section{Super analogue}\label{sec:super}

For each $\Gamma\in\gt{E}$, we define an arithmetic group
$^s\Gamma<\PSL_2(\RR)$ as follows. Recall that for each
$\Gamma\in\gt{E}$, the corresponding group $\Gamma_0$ (c.f.
\S\ref{sec:dyn:edges:thread}) can be written in the form
$G_{(a,N/a)}^{(a)}$ for some $a=a_{\Gamma}$ and $N=N_{\Gamma}$.
Define groups $^s\Gamma_0$ by setting
\begin{gather}
     ^s\Gamma_0=G_{(a,2N/a)}^{(a)}
\end{gather}
when $\Gamma_0=G_{(a,N/a)}^{(a)}$. We will arrive at the group
$^s\Gamma$ by adjoining to $^s\Gamma_0$ the appropriate ``scalings''
of the Atkin--Lehner involutions of $\Gamma$; i.e. scalings of the
cosets $\Gamma/\Gamma_0$. More precisely, a coset $W_e(N)$ (c.f.
\S\ref{sec:arith:stabs}) in $\Gamma/\Gamma_0$, for $N=N_{\Gamma}$
and $e$ an exact divisor of $N$, is scaled (i.e. sent) to
$W_{2e}(2N)$ or $W_{e}(2N)$ according as $2|e$ or not.
\begin{gather}
     ^sW_{e}(N)=\begin{cases}
               W_{2e}(2N),&\text{ if $2|e$;}\\
               W_{e}(2N),&\text{ else.}
               \end{cases}
\end{gather}
This handles the cosets of $\Gamma_0$ in $\Gamma$ in the cases that
$\Gamma\neq G_{\{2,4\}}^{(2)}$. For $\Gamma=G_{\{2,4\}}^{(2)}$ we
apply the same idea, but the notation is slightly more involved: the
unique non-trivial coset in $\Gamma/\Gamma_0$ is $g_2^{-1}W_2(2)g_2$
in this case, and we set
$^2(g_2^{-1}W_2(2)g_2)=g_2^{-1}W_{4}(4)g_2$. Now we define
$^s\Gamma$, for $\Gamma\in\gt{E}$, by setting
\begin{gather}
     ^s\Gamma=\bigcup_{X\in\Gamma/\Gamma_0}{}^s\!X.
\end{gather}
It is straightforward to check that $^s\Gamma$ is an arithmetic
subgroup of $\PSL_2(\RR)$ for each $\Gamma\in\gt{E}$. Even more than
this, each group $^s\Gamma$ also appears in Monstrous Moonshine.
Define a set $^s\gt{E}$ by setting
\begin{gather}
     ^s\gt{E}=\left\{{}^s\Gamma\mid\Gamma\in\gt{E}\right\}
\end{gather}
We may consider the diagram obtained from
(\ref{diag:intro_mdlrcorresp}) by replacing the labels
$\Gamma\in\gt{E}$ with the corresponding groups
$^s\Gamma\in{}^s\gt{E}$. In the notation of \cite{ConNorMM} and
\cite{ConMcKSebDiscGpsM}, the labeling of the affine $E_8$ Dynkin
diagram thus obtained is displayed in
(\ref{diag:super|smdlrcorresp}).
\begin{equation}\label{diag:super|smdlrcorresp}
     \xy
     (0,0)*+{2}="1A";
     (0,0)*\xycircle(2,2){-};
     (14,0)*+{4+}="2A";
     (28,0)*+{6+6}="3A";
     (42,0)*+{8+}="4A";
     (56,0)*+{10+10}="5A";
     (70,0)*+{12+}="6A";
     (80,10)*+{6\|3}="3C";
     (82,-7)*+{8\|2+}="4B";
     (94,-14)*+{4}="2B";
     {\ar@{-} "1A";"2A"};
     {\ar@{-} "2A";"3A"};
     {\ar@{-} "3A";"4A"};
     {\ar@{-} "4A";"5A"};
     {\ar@{-} "5A";"6A"};
     {\ar@{-} "6A";"3C"};
     {\ar@{-} "6A";"4B"};
     {\ar@{-} "4B";"2B"};
     \endxy
\end{equation}
Each of the groups in (\ref{diag:super|smdlrcorresp}) is attached to
a unique conjugacy class of the Monster group, curtesy of Monstrous
Moonshine. Each conjugacy class $nZ$ in $\MM$ so obtained has the
property that if $g\in nZ$ then $g^{n/2}\in 2B$. Further, each of
these conjugacy classes is a ``lift'' of some conjugacy class in the
Conway group $\Co_0$. More precisely, for each group $^s\Gamma$ in
$^s\gt{E}$ there is a Frame shape\footnote{a generalization of the
cycle notation for permutations, which encodes the eigenvalues of an
orthogonal transformation of finite order that is writable over
$\ZZ$.} $A_1^{\alpha_1}A_2^{\alpha_2}\cdots$ say, such that the eta
product
\begin{gather}
     \frac{\eta(A_1\tau)^{\alpha_1}\eta(A_2\tau)^{\alpha_2}\cdots}
          {\eta(2A_1\tau)^{\alpha_1}\eta(2A_2\tau)^{\alpha_2}\cdots}
\end{gather}
furnishes a principal modulus (a.k.a. hauptmodul) for $^s\Gamma$,
and also encodes the eigenvalues of the elements of a unique
conjugacy class of $\Co_0$.
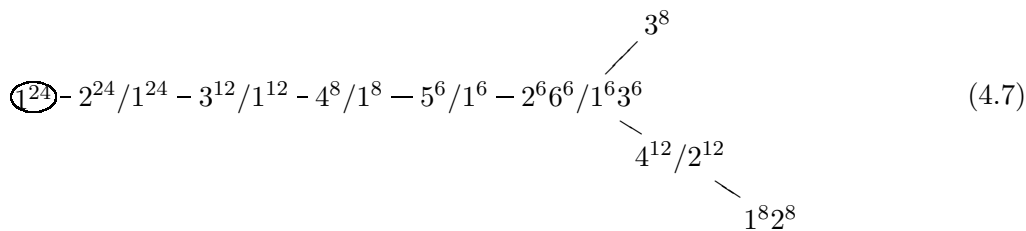
\begin{equation}\label{diag:super|conwaycorresp}
     \xy
     (0,0)*+{1^{24}}="1A";
     (0,0)*\xycircle(3,2){-};
     (12,0)*+{2^{24}/1^{24}}="2A";
     (28,0)*+{3^{12}/1^{12}}="3A";
     (42,0)*+{4^8/1^8}="4A";
     (56,0)*+{5^6/1^6}="5A";
     (73,0)*+{2^66^6/1^63^6}="6A";
     (83,10)*+{3^8}="3C";
     (86,-8)*+{4^{12}/2^{12}}="4B";
     (98,-16)*+{1^82^8}="2B";
     {\ar@{-} "1A";"2A"};
     {\ar@{-} "2A";"3A"};
     {\ar@{-} "3A";"4A"};
     {\ar@{-} "4A";"5A"};
     {\ar@{-} "5A";"6A"};
     {\ar@{-} "6A";"3C"};
     {\ar@{-} "6A";"4B"};
     {\ar@{-} "4B";"2B"};
     \endxy
\end{equation}
Diagram (\ref{diag:super|conwaycorresp}) shows the labeling obtained
by replacing the groups in (\ref{diag:super|smdlrcorresp}) with the
corresponding Frame shapes of $\Co_0$.

Evidently, (\ref{diag:super|conwaycorresp}) furnishes a
reformulation of McKay's Monstrous $E_8$ observation, in which
conjugacy classes of the Monster are replaced by conjugacy classes
of $\Co_0$.

It is amusing to observe that the highest root labeling of the
affine $E_8$ Dynkin diagram can be read off directly from the Frame
shapes in (\ref{diag:super|conwaycorresp}): consider the maximal
$A_i$ in $A_1^{\alpha_1}A_2^{\alpha_2}\cdots$ (i.e. the order of the
corresponding class in $\Co_0$). One can also predict the valency of
the vertex corresponding to any Frame shape in
(\ref{diag:super|conwaycorresp}): the valency of the vertex labeled
$A_1^{\alpha_1}A_2^{\alpha_2}\cdots$ is the number of negative
$\alpha_i$ plus $1$.

A direct link between the classes in
(\ref{diag:super|conwaycorresp}) and the principal moduli of the
groups in (\ref{diag:super|smdlrcorresp}) can be obtained by
considering the McKay--Thompson series associated to conjugacy
classes of $\Co_0$ via the action of this group on a suitably
defined vertex operator superalgebra (c.f. \cite{Dun_VACo}).

\section*{Acknowledgement}

This article is dedicated to John McKay, with gratitude and respect.
The author is grateful to Noam Elkies and Curtis McMullen for
helpful conversations


\end{document}